\documentclass{article}
\usepackage{graphicx,amsmath,amsfonts,amssymb,bm,hyperref,url,epsfig,epsf,color,fullpage,MnSymbol,mathbbol,fmtcount,algorithmic,algorithm,semtrans,caption,subcaption,multirow,amsthm,enumitem,mathrsfs,natbib,setspace,authblk}
\hypersetup{colorlinks,linkcolor={blue},citecolor={blue},urlcolor={red}}

\newtheorem{theorem}{Theorem}

\newtheorem{lemma}{Lemma}
\newtheorem{cor}{Corollary}
\newtheorem{prop}{Proposition}

\newtheorem{assumption}{Assumption}
\newtheorem{definition}{Definition}

\DeclareMathOperator{\Tr}{Tr}

\newtheorem*{remark*}{Remark}
\newtheorem*{thm*}{Theorem}
\newtheorem*{lemma*}{Lemma}
\newtheorem*{prop*}{Proposition}
\newtheorem*{assumption*}{Assumption}

\title{Weak Identification in Peer Effects Estimation}
\author{William W. Wang}
\author{Ali Jadbabaie}
\affil{Massachusetts Institute of Technology}
\date{}

\begin{document}
\maketitle
\begin{abstract}
It is commonly accepted that some phenomena are social: for example, individuals' smoking habits often correlate with those of their peers. Such correlations can have a variety of explanations, such as direct contagion or shared socioeconomic circumstances. The network linear-in-means model is a workhorse statistical model which incorporates these peer effects by including average neighborhood characteristics as regressors. Although the model's parameters are identifiable under mild structural conditions on the network, it remains unclear whether identification ensures reliable estimation in the ``infill'' asymptotic setting, where a single network grows in size. We show that when covariates are i.i.d. and the average network degree of nodes increases with the population size, standard estimators suffer from bias or slow convergence rates due to asymptotic collinearity induced by network averaging. As an alternative, we demonstrate that linear-in-sums models, which are based on aggregate rather than average neighborhood characteristics, do not exhibit such issues as long as the network degrees have some nontrivial variation, a condition satisfied by most network models.
\end{abstract}
\section{Introduction}
The study of peer effects and social interactions has grown immensely in recent years, driven by the increasing availability of social network data. Researchers seek to understand the extent to which interactions between individuals drive correlations in outcomes. For example, smoking behavior has been shown to spread through social networks, with individuals more likely to smoke if their friends do \citep{saari2014smoking}. One possible explanation for this is direct transmission of behavior: individuals may mimic their peers' behaviors due to factors such as social pressure. This is known as an \textit{endogenous} peer effect. However, there are other possible sources of correlation: friends may tend to belong in the same socioeconomic class, and individuals in the same class could behave similarly. This is known as a \textit{contextual} peer effect.

The seminal work of \citet{manski1993identification} formalized the challenge of separately identifying endogenous and contextual peer effects. In his framework, endogenous effects are defined as the influence of average peer outcomes, while contextual effects reflect the influence of average peer baseline characteristics. Manski showed that in this setting, the two effects are not separately identifiable, making it statistically impossible to distinguish the source of observed outcome correlations. He termed this the ``reflection problem.''

While Manski's work focused on dense group interactions where individuals are presumed to be affected by everyone in their group, it has been recognized that social network data can provide more detailed accounts of which interactions are actually taking place. Borrowing from ideas in the spatial econometric literature, \citet{bramoulle2009identification} proposed the following network linear-in-means model, which defines endogenous and contextual peer effects as average outcomes and baseline covariates over an individual's social connections, respectively:
\begin{align*}
    Y=\alpha 1 + X\beta + GX\delta + \rho GY + \varepsilon,
\end{align*}
where $Y\in \mathbb{R}^{n}$ are outcomes, $X\in \mathbb{R}^{n\times p}$ are baseline $p$-dimensional covariates, $G\in \mathbb{R}^{n\times n}$ is a row-normalized adjacency matrix of social interactions, and $\varepsilon\in \mathbb{R}^n$ is an error term. The parameters of interest are $\theta:=(\alpha, \beta, \delta, \rho)$, where $\delta$ and $\rho$ represent contextual and endogenous peer effects, respectively. Since the term $GY$ is endogenous, identification usually requires a valid instrument. The key insight of \citet{bramoulle2009identification} is that higher-order lags of $X$, such as $G^2X$, serve as valid instruments; consequently, if the collection $[1, X, GX,G^2X]$ are not collinear, then the peer effect parameters are identified, making Manski's non-identification result a special case where the network is complete.

Most existing applications of this model assume \textit{strong identification}, where the relevant covariance matrices remain well-conditioned in the limit.  This assumption is often justified by an asymptotic regime in which data are drawn from many small, independent networks, such as school classrooms \citep{sacerdote2011peer} or villages \citep{banerjee2013diffusion, cai2015social}. Such an assumption mirrors the “increasing-domain” asymptotics used in spatial statistics \citep{anselin2022spatial}, and standard asymptotic theory for independent and identically distributed (i.i.d.) data holds by considering each network to be an independent sample.

However, many modern social settings do not fit this framework. For example, online social media platforms typically involve a single network of interconnected users who may have a large number of connections \citep{aral2012information}. In such environments, identification strength may deteriorate as network density increases. This is particularly problematic for the linear-in-means model, where key covariates, such as neighborhood averages, become increasingly similar across individuals in dense networks. In the weak instruments literature, this loss of variation leads to slower convergence rates \citep{antoine2009efficient,hahn2002discontinuities} or non-standard limiting distributions \citep{staiger1994instrumental}.

While weak identification has been observed in simulations of peer effects \citep{bramoulle2009identification, startz2017improved}, the phenomenon remains theoretically underexplored. In this paper, we seek to answer the following question: under what natural conditions does the linear-in-means model suffer from weak identification, and how does this affect standard estimators of peer effects?

Motivated by intuition on neighborhood averages, we analyze peer effects estimation under two simple assumptions: that network density increases at any rate, and that covariates are i.i.d. We show that these two assumptions are sufficient for weak identification to occur. Specifically, we prove that ordinary least squares is generally biased, in contrast to prior results showing consistency under strong identification \citep{lee2002consistency}. On the other hand, we establish that two-stage least squares estimation is consistent for networks with density $o(\sqrt{n})$, although the estimation rate degrades as network density increases. This non-standard asymptotic behavior occurs as a result of vanishing signal in network averages and asymptotic collinearity of the design matrix.

These results naturally motivate the search for alternative models that are less vulnerable to weak identification. We consider the related linear-in-sums model, which replaces the row-normalized adjacency matrix $G$ with the unnormalized matrix $A$. We demonstrate that this model generally avoids the identification issues seen in the linear-in-means model. In particular, for random graphs, we show that as long as the underlying network model has sufficient structural variation, the model is strongly identified, and the two-stage least squares estimator converges at a standard rate. Linear-in-sums models have been applied more recently in the literature and can be interpreted as local aggregation instead of averaging \citep{ballester2006s,liu2014endogenous}. We argue that this model can be a useful alternative in discerning endogenous and contextual effects.

This paper is organized as follows: in Section \ref{sec:wi-related}, we survey some related literature on the linear-in-means model. Section \ref{sec:wi-framework} formally introduces the network linear-in-means model along with our asymptotic framework and baseline assumptions. In Section \ref{sec:wi-lim}, we state and prove our main results on weak identification in the linear-in-means model. In Section \ref{sec:wi-lis}, we show generic strong identification for the linear-in-sums model. We provide Monte Carlo simulations in Section \ref{sec:wi-simulations} and conclude in Section \ref{sec:wi-conclusion}. All omitted proofs can be found in the \hyperref[appn]{Appendix}.

\subsection{Notation}
Throughout this paper, we use the standard probability notation, such as $O_P(\cdot)$, $\Omega_P(\cdot)$, etc. We write $f(n)\ll g(n)$ if $f = o(g)$, as well as $f(n)\lesssim g(n)$ if $f(n)=O(g(n))$. We also write $\|\cdot\|_F$ for the Frobenius norm and $\|\cdot\|_2$ for the spectral norm on matrices. We further use the notation $a\vee b=\max(a,b)$ and $a\wedge b=\min(a,b)$.
\section{Related Work}\label{sec:wi-related}
\citet{bramoulle2009identification} proposed to solve the reflection problem using the network linear-in-means model. They show that in the many-group setting, endogenous and contextual effects are identified if the networks satisfy a linear independence constraint that $I, G, G^2$ are independent in the population, where $G$ is the row-normalized adjacency matrix. This identification strategy has roots in the spatial statistics literature: the network linear-in-means model is therein known as the Spatial Durbin model \citep{anselin2022spatial}, and the two-stage least-squares approach for estimating this model is due to \citet{kelejian1998generalized}.

In their simulations, \citet{bramoulle2009identification} also recognized the potential problem of weak instruments in large, dense network settings. While this issue is not widely acknowledged in the peer effects literature, it has been explored by a few authors. In particular, \citet{lee2004asymptotic,lee2007method} analyzed the properties of quasi-MLE and GMM for the linear-in-means model without contextual effects. The author recognized a case where asymptotic collinearity occurs, namely the setting of large group interactions where group size increases with the population size. In this setting, the author shows the same estimation rate $\sqrt{n/h_n}$, where $1/h_n$ is the growth rate of the entries of $G$. The rate we show under asymptotic collinearity is the same, though in our context, the collinearity is driven by local averaging instead of group interactions.

\citet{startz2017improved} also acknowledge the weak instruments problem and propose an alternative estimator based on approximating the series expansion of the endogenous term. While their estimator outperforms standard two-stage least-squares in simulation, they do not show formal guarantees.

\citet{lee2002consistency} considers estimation of spatial autoregressive models using OLS in a setup with only $X$ and $GY$ in the regression, where $X$ is presumed fixed. While there is endogeneity under this setup, the endogeneity is shown to be vanishing in magnitude if the network is dense enough, and thus OLS can be consistent at a slower rate, with error on the order of $O\left(\sqrt\frac{d}{n}\right)$. Our work contrasts with this by showing that if endogenous and contextual effects are both included in the model, then the OLS estimator is not consistent in general.

The regularity conditions in our work require the network degrees to be near-regular, a property which is satisfied in many statistical models of dense networks \citep{frieze2015introduction}. On the other hand, \citet{lee2023qml} consider estimation of spatial autoregressive models when there is a bounded amount of degree heterogeneity. They employ a novel central limit theorem to show that estimates of the model parameters are still typically consistent and asymptotically normal at a $\sqrt{n}$ rate, so long as the fraction of ``dominant units'' is bounded. The results in our current work could extend to such settings. However, it is possible that estimators would have non-standard asymptotic distributions when the regressors are considered random, since the limiting variance depends on the specific $X$ value of the dominant units.

In concurrent work, \citet{hayes2024peer} also derive results on the asymptotic behavior of OLS for the full linear-in-means model with endogenous and contextual effects and random regressors. Under an additional subgamma condition on the covariates and error terms, they show that if the degree of the graph is $\Omega(\log n)$ then $GX$ and $GY$ converge uniformly to constants, and they further show that the OLS estimation rates for the linearly dependent terms of the model depend on $\Tr(G^TG)$. Our results for OLS agree with and refine some of their findings, and we are able to provide a more specific description of the reason for inconsistency (i.e., by decomposing the error into a sum of an asymptotically normal term plus bias). Additionally, we are able to provide new results for two-stage least squares. Our proof method is somewhat conceptually simpler than that of Hayes and Levin: while their work produces an approximate singular value decomposition, we directly compute the matrix inverse, which is tractable by invoking the Frisch-Waugh-Lovell theorem. By carefully considering the scaling of the moments as the network structure varies, we can derive estimation rates.

The linear-in-sums model has become more popular in recent years. \citet{liu2014endogenous} formulate a model which incorporates both endogenous average and aggregate components, which they estimate jointly. Linear-in-sums models have seen applications in treatment effect estimation~\citep{cai2015social,chandrasekhar2011econometrics} and in network causal inference as a parametric model of interference \citep{chin2019regression}. Our work provides some additional support in favor of these models by arguing that they generically do not suffer from weak identification.
\section{The Network Linear-in-Means Model}\label{sec:wi-framework}
\subsection{Asymptotic Framework}
For the peer effects model we describe, the observed data includes nodewise covariates $X_i\in \mathbb{R}^p$ where $p$ is the number of covariates, as well as nodewise outcomes $Y_i\in \mathbb{R}$ for each of $n$ individuals. We collect these individual covariates and outcomes into a matrix $X\in \mathbb{R}^{n\times p}$ and vector $Y\in \mathbb{R}^n$ respectively. Furthermore, we observe a binary symmetric adjacency matrix of social interactions, $A\in [0,1]^{n\times n}$. The asymptotics in this work are taken with respect to the population size, $n$: that is, if $X^{(n)}$, $Y^{(n)}$, and $A^{(n)}$ are the covariates, outcomes, and adjacency matrix for a population of size $n$, then we observe a sequence $(X^{(n)},Y^{(n)},A^{(n)})_{n\in \mathbb{N}}$, and asymptotics are taken in reference to this sequence. We will generally omit the superscript for clarity.
\subsection{Linear-in-Means Model}\label{section:network-asymptotics}
For an adjacency matrix $A$, we denote $D$ to be the diagonal matrix where the $i$-th entry on the diagonal is $d_i:=\sum_{j=1}^n A_{ij}=(A1)_i$, where $1$ is the $n$-dimensional vector of ones. We then define $G=D^{-1}A$ to be the row-normalized version of $A$. To ensure this is well-defined if $d_i=0$, we re-define $d_i^{-1}:=1$ if $d_i=0$,

With this notation, the classical linear-in-means model \citep{bramoulle2009identification} is defined as follows:
\begin{align*}
    Y=\alpha 1 + X\beta + GX\delta + \rho GY + \varepsilon,
\end{align*}
where $\varepsilon\in \mathbb{R}^n$ is an unobserved error term and $(\alpha,\beta,\delta,\rho)$ are parameters of the model. $\alpha\in \mathbb{R}$ and $\beta\in \mathbb{R}^p$ are the parameters of a standard linear model, while $\delta\in \mathbb{R}^p$ and $\rho\in \mathbb{R}$ are network effect parameters modeling contextual and endogenous peer effects, respectively. In this paper, we will take $p=1$ so that $\delta$ is a scalar, and we will primarily be interested in estimation of the network parameters $\delta$ and $\rho$, rather than on estimation of the coefficients $\alpha,\beta$.

To simplify notation, we make the following definitions for instruments $Z$ and endogenous variables $W$, respectively:

\begin{align*}
    Z&:=(1,X,GX,G^2X),\\
    W&:=(1,X,GX,GY).
\end{align*}
If $\theta:= (\alpha,\beta,\delta,\rho)$, then we can rewrite the peer effects model as
\begin{align*}
    Y=\theta W+\varepsilon.
\end{align*}
The OLS estimator is defined as
$$\hat{\theta}_{OLS} := (W^TW)^{-1}(W)^TY,$$
and the two-stage least squares estimator is then
$$\hat{\theta}_{2SLS} := ((P_ZW)^TP_ZW)^{-1}(P_ZW)^TY,$$
where $P_Z=Z(Z^TZ)^{-1}Z^T$ is the orthogonal projection onto the column space of $Z$. Since we are working in the just-identified regime, this is equivalent to the IV estimator,
$$\hat{\theta}_{2SLS} := (Z^TW)^{-1}Z^TY.$$
We will also introduce the following notation for some important network-related quantities:
\begin{align}\label{def:SH}
    S=S(\rho):=(I-\rho G),\quad H=H(\theta):=(\beta I+\delta G)S^{-1}.
\end{align}
In this notation,
\begin{equation}\label{eq:rf-y}
    Y=\alpha S^{-1}1+HX+S^{-1}\varepsilon.
\end{equation}
We now describe our key assumptions on $X,\varepsilon$, and the network.
\subsection{Regularity Assumptions}
Throughout our analysis, we impose the following regularity conditions, which are standard in the literature \citep{kelejian1998generalized}:
\begin{assumption}[Error terms]\label{assumption:errors}
    The $\varepsilon_i$ are exogenous (i.e., $E[\varepsilon_i\mid A,X]=0$), homoskedastic with variance $\sigma_\varepsilon^2$, and have bounded fourth moment.
\end{assumption}
\begin{assumption}[Regularity of $G$]\label{assumption:network-reg}
    The row-normalized adjacency matrix $G$ as well as $S(\rho)^{-1}$ have uniformly bounded column sums, i.e., 
    $$\max_i (G^T1)_i=O(1),$$
    and similarly for $S(\rho)^{-1}$.
\end{assumption}
\begin{assumption}[Convergence of appropriate moments]\label{assumption:moments}
    $\frac{1}{n}Z^TZ\rightarrow_P \Gamma_{ZZ}$, $\frac{1}{n}Z^TW\rightarrow_P \Gamma_{ZW}$, and $\frac{1}{n}W^TW\rightarrow_P \Gamma_{WW}$ for some matrices $\Gamma_{ZZ}$, $\Gamma_{ZW}$, and $\Gamma_{WW}$. Additionally, $Z^TZ$, $Z^TW$, and $W^TW$ are invertible with probability 1 for each $n$.
\end{assumption}
The homoskedasticity assumption on $\varepsilon$ is for convenience but can be relaxed. The assumption that $G$ and $S(\rho)^{-1}$ have bounded column sums is used to ensure the law of large numbers holds and is almost always enforced in the literature, though there is some recent work which allows it to be relaxed \citep{lee2023qml}. Some assumption on the growth of the column sums of these matrices is necessary to avoid ``concentration of influence'' phenomena, where moments such as $\frac{1}{n}1^TGX$ would be dominated by the value of a single well-connected dominant unit; see, relatedly, discussions of the implications of this phenomenon for network economies and governance in \citet{acemoglu2012network, halpern2021defense}.

Assumption \ref{assumption:moments} is used to argue that appropriate sample moments converge; however, contrary to much existing literature, we do \textit{not} make the common assumption that the relevant limiting moment matrices, such as $\Gamma_{ZW}$, are invertible. 
\subsection{Network Conditions}
Our main assumption for the network sequence $(A^{(n)})_{n\in \mathbb{N}}$ is the following near-degree regularity condition:
\begin{assumption}[Near-Degree Regularity]\label{assumption:regular}
    There exist constants $c_1,c_2$ and a scaling sequence $d=d(n)$ so that for all $n$ sufficiently large,
    $$c_1d(n)\leq d_i^{(n)}\leq c_2d(n),$$
    uniformly for all $i=1,...,n$, where $d_i^{(n)}$ is the degree of node $i$ in $A^{(n)}$, that is, $d_i^{(n)}=\sum_{j=1}^n A_{ij}^{(n)}$. Furthermore, $1 \ll d \ll n$.
\end{assumption}
This assumption is strictly stronger than Assumption \ref{assumption:network-reg}. For example, Assumption \ref{assumption:network-reg} allows networks where nodes of degree one are attached to a clique of growing size, while our Assumption \ref{assumption:regular} would require that these low-degree nodes match the degree of the clique asymptotically. Nevertheless, it is well-known to be satisfied by a number of graph models, such as random graphs with degree $\omega(1)$ \citep{frieze2015introduction}, and is also assumed in some existing analyses of spatial autoregressive models \citep{lee2002consistency}.

An important consequence of Assumption \ref{assumption:regular} are the following bounds, which will govern the rates of convergence for the estimators we analyze:
\begin{prop}\label{prop:rates}
    Under Assumption \ref{assumption:regular}, $\|G\|_F^2=\Theta\left(\frac{n}{d}\right)$. In addition, $\frac{n}{d^2}\lesssim \|G^2\|_F^2\lesssim \frac{n}{d}$. 
\end{prop}
The last assumption we will use for our analysis is an i.i.d. assumption on the covariates: 
\begin{assumption}[IID covariates]\label{assumption:iid}
    The covariates $X_i$ are univariate independent and identically distributed random variables with mean $\mu$, variance $\sigma^2$, bounded fourth moment, and are independent of $A$.
\end{assumption}
This assumption appears benign in general and would hold if $X_i$ were a randomly administered treatment in a network experiment  \citep{cai2015social,banerjee2013diffusion}. However, when these two assumptions hold, $\Gamma_{ZW}$ is provably low-rank:
\begin{prop}\label{prop:low-rank}
    Under Assumptions \ref{assumption:regular} and \ref{assumption:iid}, $\Gamma_{ZW}$ has the rank-two structure $ab^T+E$, where $a=(1,\mu,\mu,\mu)^T$, $b=(1,\mu,\mu,\mu')^T$, $\mu'=\frac{\alpha+\mu(\beta+\delta)}{1-\rho}$, and $E=\mathrm{diag}(0,\sigma^2,0,0)$.
\end{prop}
The proposition holds via simple computation of the limiting moment matrix. The rank-deficient limit is caused by asymptotic collinearity between $1$, $GX$, and $GY$. Intuitively, as the neighborhood of each individual is denser, more of the variation in $X$ which is useful for identification gets averaged out, so that in the limit, $GX$ approaches a constant. \citet{hayes2024peer} provide a uniform convergence result of this type.

While standard arguments for least-squares consistency cannot be applied directly due to the rank deficiency, it is not immediately clear that standard estimators are inconsistent. Indeed, in similar non-network settings where an invertible covariance matrix has a singular limit, standard estimators are often consistent but at a slower rate \citep{knight2008shrinkage,phillips2016inference}. In the following section, we analyze both OLS and 2SLS estimators of $\theta$ and show that while OLS is generally inconsistent, 2SLS can be consistent provided the network is not too dense.
\section{Weak Identification in Linear-in-Means}\label{sec:wi-lim}
In this section, we analyze the behavior of the OLS and 2SLS estimators for the linear-in-means model. We maintain the assumptions defined in the previous section. In addition, our results will require the following mild regularity assumption on the sequence of network moments:
\begin{assumption}\label{assumption:network-mom}
     For $A$ and $B$ chosen from the set $\{I,S^{-1},H\}$ where $S$ and $H$ are defined in (\ref{def:SH}), the limit $\lim\limits_{n\rightarrow\infty} \frac{1}{\|G\|_F^2}\Tr((GA)^TGB)=:m_{GA,GB}$ exists. Additionally, $m_{I,GS^{-1}}:=\lim\limits_{n\rightarrow\infty} \frac{1}{\|G\|_F^2}\Tr(GS^{-1})$ exists.
\end{assumption}
The term $\Tr((GA)^TGB)$ is a signal term which arises from the expected covariance matrix. It is $O\left(\frac{n}{d}\right)$, which can be seen from the following:
\begin{align*}
    \Tr((GA)^TGB)&\leq \|GA\|_F\|GB\|_F\\
    &\leq \|G\|_F\|A\|_2\|G\|_F\|B\|_2\lesssim \|G\|_F^2=O\left(\frac{n}{d}\right),
\end{align*}
using Cauchy-Schwarz, the inequality $\|AB\|_F\leq \|A\|_F\|B\|_2$, Lemma \ref{lemma:col-sums} in the \hyperref[appn]{Appendix}, and Proposition \ref{prop:rates} under our near-degree regularity assumption. In many cases which we describe below, it is also $\Omega(\frac{n}{d})$, so $\frac{1}{\|G\|_F^2}\Tr((GA)^TGB)=\Theta(1)$. In this case, this assumption is simply asserting that this boundedness translates to a deterministic limit. 

Since $\|G\|_F^2=\Theta(\frac{n}{d})$, we can alternatively assume that $\frac{d}{n}\|G\|_F^2=\frac{d}{n}\sum_{j=1}^n \frac{1}{d_j}$ converges, and that $\frac{d}{n}\Tr((GA)^TGB)$ converges. In this case, since $\frac{1}{n}\sum_{j=1}^n \frac{1}{d_j}$ is the reciprocal of the harmonic mean of degree sequence, the assumption that $\frac{d}{n}\|G\|_F^2$ converges is equivalent to a natural assumption that the harmonic mean is asymptotically equivalent to a constant factor times $d$. For example, it excludes sequences such as Erd\H{o}s--R\'enyi random graphs with a randomly chosen connection probability for each $n$. We choose to state the assumption using normalization by $\|G\|_F^2$ instead of the network-independent quantity $\frac{n}{d}$ since it will generalize more easily to the analysis of two-stage least squares, which has a rate that depends on the structure of the network beyond the first-order degrees.

In general, since the terms $\Tr((GA)^TGB)$ are upper bounded in magnitude by $O(\frac{n}{d})$, Assumption \ref{assumption:network-mom} is a mild condition which requires that this upper bound translates to deterministic limits. Depending on the parameter values, some of these limits could be zero. For example, if $\beta = 0$, then 
\begin{align*}
    \Tr(G^TGH)&=\Tr(\delta G^TG^2S^{-1})\leq |\delta|\|G^2\|_F\|G^TS^{-1}\|_F\\
    &\leq |\delta|\|G^2\|_F\|G\|_F\|S^{-1}\|_2\\
    &\lesssim \|G^2\|_F\|G\|_F.
\end{align*}
In random graphs, $\|G^2\|_F=O_P(1\vee \frac{\sqrt{n}}{d})$ and $\|G\|_F=\Theta_P(\sqrt\frac{n}{d})$, so $\Tr(G^TGH)=o_P(\frac{n}{d})$ and thus $m_{G,GH}=0$. On the other hand, the following lemma records which limits are guaranteed to be nonzero, as well as their sign:

\begin{lemma}\label{lemma:nonzero-moments}
    $m_{G,G}$, $m_{G,GS^{-1}}$, $m_{GS^{-1},GS^{-1}}$, and $m_{I,GS^{-1}}$ are nonzero. Additionally, if $\beta\neq 0$, $m_{G,GH}$ and $m_{GH,GH}$ are nonzero.
\end{lemma}
\subsection{Ordinary Least Squares Estimation of Linear-in-Means}
We now analyze the behavior of the OLS estimator of the network effects. Recall that this estimator is defined as 
\begin{equation*}
    \hat\theta_{OLS}=(W^TW)^{-1}W^TY.
\end{equation*}

Since we are primarily interested in the network effects $\delta$ and $\rho$, we partial out the intercept and $X$, which allows us to concentrate our analysis on just two parameters. Specifically, denote $W_1=(1,X)$, $W_2=(GX,GY)$, $P_1=W_1(W_1^TW_1)^{-1}W_1^T$, $M=I-P_1$, $\tilde{W}_2=MW_2$, $\theta_1=(\alpha, \beta)$, $\theta_2=(\delta, \rho)$ and $\hat\theta_2=(\hat{\delta}_{OLS},\hat{\rho}_{OLS})$. Then by the Frisch-Waugh-Lovell theorem \citep{ding2024linear}, we can write
\[
\hat{\theta}_2 -\theta_2 = (\tilde{W}_2^T\tilde{W}_2)^{-1}\tilde{W}_2^T\varepsilon.
\]
To formally state our main result, we introduce the following quantities involving network moments which are well-defined using Assumption \ref{assumption:network-mom}:
\begin{align}\label{eq:covariance-ols}
    \tilde\Gamma_{WW}&=\begin{pmatrix}
        \sigma^2 m_{G,G}&\sigma^2 m_{G,GH}\\
        \cdot &\sigma^2 m_{GH,GH}+\sigma_\varepsilon^2m_{GS^{-1},GS^{-1}}
    \end{pmatrix},\\
    \tilde\Sigma_{WW}&=\sigma_\varepsilon^2\begin{pmatrix}
        \sigma^2 m_{G,G}&\sigma^2 m_{G,GH}\\
        \cdot &\sigma^2 m_{GH,GH}
    \end{pmatrix},
\end{align}
and $\bar\kappa:=\det(\tilde\Gamma_{WW}).$ With this notation defined, we have the following theorem, which decomposes the regression error into the sum of an asymptotically normal term and a residual bias term:
\begin{theorem}\label{thm:ols-formal}
    Under assumptions \ref{assumption:errors}--\ref{assumption:network-mom},
    \begin{align*}
        \hat\theta_2-\theta_2&=
        V_n+b_n,
    \end{align*}
    where $\|G\|_FV_n\rightarrow N(0,\tilde\Gamma_{WW}^{-1}\tilde\Sigma_{WW}\tilde\Gamma_{WW}^{-1})$, and 
    $$b_n\rightarrow_P \sigma_\varepsilon^2\tilde\Gamma_{WW}^{-1}\begin{pmatrix}
        0\\m_{I,GS^{-1}}
    \end{pmatrix}.$$
    If $\rho> 0$, then the bias term for $\hat{\rho}$ is strictly positive, and so $\hat{\rho}$ is inconsistent and upward biased. If in addition, $\beta\neq 0$, then the bias term for $\hat{\delta}$ is nonzero as well. 
\end{theorem}
Since $\|G\|_F=\Theta(\sqrt\frac{n}{d})$ as asserted in Proposition \ref{prop:rates}, this theorem decomposes the error into the sum of an an asymptotically normal component with variance $O(\frac{d}{n})$ and a typically non-vanishing bias. The non-convergence result contrasts with the results on OLS derived by \citet{lee2002consistency}. That work assumes there is no collinearity among the regressors, and shows that OLS converges at the standard rate $\frac{1}{\sqrt{n}}$ if $d=\Omega(n^{1/2})$, or at a slower rate $O_P\left(\sqrt\frac{d}{n}\right)$ if $1\ll d\ll n^{1/2}$. The key driver of consistency is the vanishing of endogeneity when the network density increases. In our setting, we still have vanishing of endogeneity, but the baseline rate of convergence is slower due to the collinearity between endogenous and contextual effects. Thus the increasing degree of the graph simultaneously alleviates endogeneity and induces collinearity, which produces bias. It is not clear if there is a procedure to correct this bias, as it depends on the true value of $\rho$ through $m_{I,GS^{-1}}$.

Our results are qualitatively similar to those of \citet{hayes2024peer}, who analyze OLS for the linear-in-means model under very similar assumptions. They show that if there is a nonzero endogenous effect, then $\hat\rho$ is inconsistent and $\hat\delta$ has a convergence rate lower bounded by $\Omega\left(\sqrt\frac{d}{n}\right)$. Our results provide more information on the bias and variance of the estimator. In particular, we show that inconsistency is due to the convergence to a bias term of constant order, and that the stochastic component of the error is asymptotically normal with a slower rate $\Theta_P\left(\sqrt{\frac{d}{n}}\right)$. Additionally, our result provides a more refined description of the behavior of $\hat\delta$, namely that if $\rho$ and $\beta$ are nonzero, then $\hat\delta$ is inconsistent as well.

Additionally, we can determine the sign of the bias for $\rho$: since $m_{G,G}$ is non-negative and $m_{I,GS^{-1}}$ is of the same sign as $\rho$ from the proof of Lemma \ref{lemma:nonzero-moments}, it follows that the limiting bias for $\hat\rho$, $\frac{1}{\kappa}\sigma^2\sigma_\varepsilon^2m_{I,GS^{-1}}m_{G,G}$, is of the same sign as $\rho$, meaning that the peer effect is overestimated in magnitude.

In Section~\ref{dis:ols-consistency} of the \hyperref[appn]{Appendix}, we discuss various cases where OLS is consistent, as implied by our results. These include the cases where $\rho=0$, or when $\beta=0$ and the network has a sufficiently small triangle density. Such examples further demonstrate the complexity of linear-in-means estimation under weak identification, as the consistency and estimation rates depend on both the underlying parameters and network structure.

In summary, this theorem provides a precise description of the behavior of OLS when contextual and endogenous effects are assumed to both be present in the model. The OLS estimator is upward-biased if there is an endogenous effect, and this bias is non-vanishing. The estimator generally converges to this biased quantity at a rate $\Theta_P\left(\sqrt\frac{d}{n}\right)$.
\begin{proof}[Proof of Theorem \ref{thm:ols-formal}]
    To analyze the error of this estimator, note that OLS ignores the endogeneity present in the model. Similar to the analysis by \citet{lee2002consistency} in the strongly identified case, we will analyze the contributions from the exogenous variables and endogenous variables separately.

Recalling the expression for $Y$ in Equation (\ref{eq:rf-y}), we can write $\tilde{W}_2=\tilde{W}_2^{exog}+\tilde{W}_2^{endog}$, where $\tilde{W}_2^{exog}=(MGX,MG(\alpha S^{-1}1+HX))^T$ and $\tilde{W}_2^{endog}=(0,MGS^{-1}\varepsilon)^T$.
The resulting estimator error is then
$$\hat{\theta}_2-\theta_2=(\tilde{W}_2^T\tilde{W}_2)^{-1}(\tilde{W}_2^{exog})^T\varepsilon +(\tilde{W}_2^T\tilde{W}_2)^{-1}(\tilde{W}_2^{endog})^T\varepsilon.$$
We can now proceed to analyze the exogenous and endogenous components separately.
\subsubsection{OLS Exogenous Component}
We start by demonstrating asymptotic normality for the exogenous component of the error. To this end, we first show a convergence result for an appropriately rescaled version of the covariance matrix, $(\tilde{W}_2^T\tilde{W}_2)^{-1}$. First, defining $V=X-\mu$ as the stochastic component of $X$, note that $MGX=MGV$ and $MGY=M(GHV+GS^{-1}\varepsilon)$, i.e., the mean components are removed. Writing $e_1(A,B)=V^TA^TMBV$, $e_2(A,B)=\varepsilon^TA^TMB\varepsilon$, $v(A,B)=V^TA^TMB\varepsilon$, and $v'(A,B)=v(A,B)+v(B,A)$, we have
\begin{align*}
\frac{1}{\|G\|_F^2}\tilde{W}_2^T\tilde{W}_2 =\frac{1}{\|G\|_F^2}\begin{pmatrix}
    e_1(G,G) & e_1(G,GH)+v(G,GS^{-1}) \\
    \cdot & 
        e_1(GH,GH)+v'(GH,GS^{-1})+e_2(GS^{-1},GS^{-1})
\end{pmatrix}
\end{align*}
The following result states that this matrix converges to its population counterpart. Its proof is based on a law of large numbers for quadratic forms of network-related matrices, provided in Section~\ref{lln} of the \hyperref[appn]{Appendix}.
\begin{prop}\label{cor1}
    If $1\lesssim d\lesssim n$, then
    $\frac{1}{\|G\|_F^2}\tilde{W}_2^T\tilde{W}_2\rightarrow_P \tilde\Gamma_{WW}$, where $\tilde\Gamma_{WW}$ is defined in Equation~(\ref{eq:covariance-ols}). In addition, $\frac{1}{\|G\|_F^2}e_2(I,GS^{-1})\rightarrow_P \sigma_\varepsilon^2m_{I,GS^{-1}}$.
\end{prop}
With this convergence result shown, we now need to argue that the inverse exists. Letting $\kappa=\det \left(\frac{1}{\|G\|_F^2}\tilde{W}_2^T\tilde{W}_2\right)$, we have the inverse,
\[
\left(\frac{1}{\|G\|_F^2}\tilde{W}_2^T\tilde{W}_2\right)^{-1}=\frac{1}{\kappa}\begin{pmatrix}
e_1(GH,GH)+v'(GH,GS^{-1})+e_2(GS^{-1},GS^{-1})
&\cdot\\
    -(e_1(G,GH)+v(G,GS^{-1})) & e_1(G,G)
\end{pmatrix}.
\]
By continuity, Proposition \ref{cor1} implies that $\kappa\rightarrow_P \bar\kappa:=\det(\Gamma_{WW})$. For the limit of the inverse to make sense, we need to argue that $\bar\kappa$ is lower bounded. We have
\begin{align*}
    \bar\kappa &= \sigma^2\sigma_\varepsilon^2 m_{G,G}m_{GS^{-1},GS^{-1}}+\sigma^4 (m_{G,G}m_{GH,GH}-m_{G,GH}^2).
\end{align*}
By Cauchy-Schwarz applied to the standard inner product on matrices, the second term is non-negative. The first term is nonzero, as shown in Lemma \ref{lemma:nonzero-moments}. Thus $\Gamma_{WW}^{-1}$ is well-defined and we have
\begin{align*}
    \left(\frac{1}{\|G\|_F^2}\tilde{W}_2^T\tilde{W}_2\right)^{-1}\rightarrow_P \Gamma_{WW}^{-1}=\frac{1}{\kappa}\begin{pmatrix}
        \sigma^2 m_{GH,GH}+\sigma_\varepsilon^2m_{GS^{-1},GS^{-1}}&-\sigma^2 m_{G,GH}\\
        \cdot & \sigma^2 m_{G,G}
    \end{pmatrix}.
\end{align*}
We can now complete our analysis of the exogenous error component. In our notation, $(\tilde{W}_2^{(exog})^T\varepsilon=(v(G,I),v(GH,I))^T$, and so
\begin{align*}
    (\tilde{W}_2^T\tilde{W})^{-1}(\tilde{W}_2^{exog})^T\varepsilon
&=(\Gamma_{WW}+o_P(1))\frac{1}{\|G\|_F^2}\begin{pmatrix}
    v(G,I)\\ v(GH,I)
\end{pmatrix}.
\end{align*}
The following limit theorem establishes the behavior of $v$:
\begin{prop}\label{prop:ols-clt}
    \begin{align*}
        \frac{1}{\|G\|_F}\begin{pmatrix}
    v(G,I)\\ v(GH,I)
\end{pmatrix}\rightarrow_D N(0, \tilde\Sigma_{WW}).
    \end{align*}
\end{prop}
Combining, we obtain
\begin{align*}
    \|G\|_F\left((\Gamma_{WW}+o_P(1))\begin{pmatrix}
    v(G,I)\\ v(GH,I)
\end{pmatrix}\right)\rightarrow N(0, \tilde\Gamma_{WW}^{-1}\tilde\Sigma_{WW} \tilde\Gamma_{WW}^{-1}).
\end{align*}
It follows that exogenous error component is asymptotically normal with a convergence rate $O_P\left(\sqrt\frac{d}{n}\right)$. 
\subsubsection{OLS Endogenous Component}
To analyze the bias induced by endogeneity, we write
\begin{align*}
    (\tilde{W}_2^T\tilde{W}_2)^{-1}(\tilde{W}_2^{endog})^T\varepsilon&=\left(\frac{1}{\|G\|_F^2}\tilde{W}_2^T\tilde{W}_2\right)^{-1}\frac{1}{\|G\|_F^2}\begin{pmatrix}
    0\\ e_2(I,GS^{-1})
\end{pmatrix}\\
&=\frac{1}{\kappa}\left(\frac{1}{\|G\|_F^2}\right)^2e_2(I,GS^{-1})\begin{pmatrix}
    -(e_1(G,GH)+v(G,GS^{-1}))\\e_1(G,G)
\end{pmatrix}.
\end{align*}
As demonstrated in Proposition \ref{cor1}, $\text{plim}\frac{1}{\|G\|_F^2}e_2(I,GS^{-1})=\sigma^2_\varepsilon m_{I,GS^{-1}}$, and by Lemma \ref{lemma:nonzero-moments}, this limit is strictly nonzero when $\rho\neq 0$. Therefore the limiting bias for $\rho$ is strictly nonzero. There is generally non-vanishing bias for the contextual effect $\delta$ when $\beta\neq 0$ as well via the same argument. 
\end{proof}
\subsection{Two-Stage Least Squares Estimation of Linear-in-Means}
For 2SLS estimation, we have instruments $Z=(1,X,GX,G^2X)$. The endogenous variables $W$ are projected onto the instruments via the projection matrix $P_Z:=Z(Z^TZ)^{-1}Z^T$, resulting in the 2SLS estimator,
\begin{align*}
    \hat{\theta}_{2SLS}&=\left((P_ZW)^TP_ZW\right)^{-1}(P_ZW)^TY.
\end{align*}
Since we are working in the just-identified case where the number of instruments equals the number of endogenous variables, this is equivalent to the IV estimator, that is,
\begin{align*}
    \hat{\theta}_{2SLS}&=(Z^TW)^{-1}Z^TY.
\end{align*}

A similar strategy by partialing out $(1,X)$ simplifies the analysis. The following extension of the Frisch-Waugh-Lovell theorem for linear instrumental variable estimation was first derived by \citet{giles1984instrumental} and expanded upon by \citet{basu2023yule}:
\begin{lemma}[Adapted from \citet{basu2023yule}]
    The error of the 2SLS estimator for $\theta_2=(\delta, \rho)$ is given by
    $$\hat{\theta}_2-\theta_2=(\tilde{Z}_2^T\tilde{W}_2)^{-1}\tilde{Z}_2^T\varepsilon,$$
    where $\tilde{W}_2=MW_2$, $\tilde{Z}_2=MZ_2$, and $M=I-Z_1(Z_1^TZ_1)^{-1}Z_1^T$.
\end{lemma}
To analyze the two-stage least squares estimator, we will need an extra assumption on the strength of identification. Specifically, for 2SLS, identification is known to depend on the level of intransitivity in the network - for example, the identification result from \citet{bramoulle2009identification} requires the existence of triplets of nodes which do not form a connected triangle. We introduce a slightly stronger assumption which essentially posits a lower bound on the density of such triplets. Let $C_3$ and $C_4$ denote the sets of $3$-cycles (i.e., triangles) and 4-cycles in the graph, respectively, where 4-cycles can either be simple cycles or three-node cycles of the forms $i\rightarrow j\rightarrow i \rightarrow k\rightarrow i$ or $i\rightarrow j\rightarrow k \rightarrow j\rightarrow i$. We assume the following:
\begin{assumption}\label{assumption:triangles}
    $$\frac{|C_3|}{|C_4|}=o\left(\frac{1}{d}\right).$$
\end{assumption}
This assumption limits the density of complete triangles in the network; it rules out networks such as unions of fully connected graphs on $d$ nodes, but holds trivially for bipartite graphs. In random graphs with expected degree $\Theta(d)$ where $d\rightarrow \infty$, the assumption is satisfied when the degree is sufficiently small. In particular, for random graphs in this asymptotic regime, standard computations of expectation and variance of $|C_3|$ show that $|C_3|=\Theta_P(E\left[|C_3|\right])=\Theta(d^3)$. Similarly, by considering the types of 4-cycles, it follows that $|C_4|=\Theta_P(nd^2+d^4)$. When $d\ll \sqrt n$, the size of $C_4$ is governed by the $\Theta(nd^2)$ term and is thus the assumption is satisfied. If $d$ is too large, i.e., $d\gtrsim \sqrt{n}$, then the graph has too high a concentration of triangles, and the assumption is not satisfied.

Assumption \ref{assumption:triangles} is related to the \textit{global clustering coefficient}, a popular measure of clustering in social network analysis \citep{wasserman1994social}. Denoting $C_3^{open}$ to be the set of ordered triples $i,j,k$ where $i,j,k$ are distinct and two edges exist between the three nodes, the global clustering coefficient is defined as
$$\frac{|C_3|}{|C_3|+|C_3^{open}|}.$$
When the degrees are growing, boundedness of the global clustering coefficient implies Assumption \ref{assumption:triangles}:
\begin{prop}
    If the global clustering coefficient is $o\left(\frac{1}{d}\right)$, then Assumption \ref{assumption:triangles} is satisfied.
\end{prop}
\begin{proof}
    If $\frac{|C_3|}{|C_3|+|C_3^{open}|}=o\left(\frac{1}{d}\right)$, then
    \begin{align*}
        \frac{|C_3|+|C_3^{open}|}{|C_3|}=1+\frac{|C_3^{open}|}{|C_3|}=\omega(d),
    \end{align*}
    which shows that 
    $$\frac{|C_3|}{|C_3^{open}|}=o\left(\frac{1}{d}\right).$$
    Since all open triples in $C_3^{open}$ correspond to 3-node 4-cycles, $|C_3^{open}|\leq |C_4|$, which proves the claim.
\end{proof}
Aside from this new identification assumption, we will also need an analogue of Assumption~\ref{assumption:network-mom} on the convergence of network quantities involving $G^2$:
\begin{assumption}\label{assumption:network-mom2}
The following limits exist:
\begin{align*}
\frac{1}{\|G^2\|_F^2} \left( \Tr((G^2)^T G H) - \frac{1}{n} \Tr(G^2) \Tr(G H) \right) &\rightarrow m'_{G^2, GH}, \\
\frac{1}{\|G^2\|_F^2} \left( \Tr((G^2)^T G^2) - \frac{1}{n} \Tr(G^2)^2 \right) &\rightarrow m'_{G^2, G^2}, \\
\frac{\|G^2\|_F}{\|G\|_F} &\rightarrow \eta,
\end{align*}
for appropriate constants $m'_{G^2, GH}$, $m'_{G^2, G^2}$, and $\eta$.
\end{assumption}
As discussed previously, this assumption is a mild regularity condition which enforces that limiting network moments exist. The additional subtracted term corresponds to variation that is lost from the partialling out step of the proof. It is typically constant, as we argue in Section \ref{wi-app-assumption} of the \hyperref[appn]{Appendix}. The term $\eta$ is generally zero, except for the special case $\|G^2\|_F=\Omega(\sqrt\frac{n}{d})$, as is the case for complete bipartite graphs.

Similar to the OLS proof, this assumption allows us to define the limiting quantities which will be relevant:
\begin{align}\label{eq:covariance-2sls}
    \tilde\Gamma_{ZW}^{-1}=\begin{pmatrix}
       \frac{\eta}{\sigma^2}&-\frac{1}{\sigma^2}\frac{m_{G,GH}}{m_{G,G}m'_{G^2,GH}}\\
       0&\frac{1}{\sigma^2}\frac{1}{m'_{G^2,GH}}
    \end{pmatrix},\quad \tilde\Sigma_{ZZ}:=\sigma_\varepsilon^2\begin{pmatrix}
        \sigma^2 m_{G,G}&0\\
        0 &\sigma^2 m_{G^2,G^2}
    \end{pmatrix}.
\end{align}
Here, the quantities in the denominators are guaranteed to be nonzero by the same logic as the proof of Lemma \ref{lemma:nonzero-moments}. Additionally, note that $\tilde\Gamma_{ZW}^{-1}$ is rank-one, which is a consequence of the asymptotic collinearity between $G$ and $G^2$. 

We can now state a positive result for two-stage least squares:
\begin{theorem}\label{thm:2SLS}
    Under the assumptions of Theorem \ref{thm:ols-formal} and the additional Assumptions \ref{assumption:triangles} and~\ref{assumption:network-mom2}, if $\|G^2\|_F\rightarrow \infty$, then the two-stage least squares estimator satisfies
    \begin{align*}
        \|G^2\|_F(\hat{\theta}_2-\theta_2)\rightarrow N(0,\tilde\Gamma_{ZW}^{-1}\tilde\Sigma_{ZZ}\tilde\Gamma_{ZW}^{-T}).
    \end{align*}
\end{theorem}
\begin{cor}\label{cor:2SLS}
    Under the assumptions of Theorem \ref{thm:2SLS}, if $G$ is a random graph with $d\ll \sqrt{n}$, then $\hat{\rho}$ and $\hat{\delta}$ converge at a slower rate $\frac{d}{\sqrt{n}}$. If $G$ is a disjoint union of complete bipartite d-regular graphs, then $\hat{\rho}$ and $\hat{\delta}$ converge at a rate $\sqrt{\frac{d}{n}}$.
\end{cor}
We can see that the convergence rate can be different for networks with the same degree. In the best case, the network can be a disjoint union of complete bipartite $d$-regular graphs, which achieves a convergence rate of $O\left(\sqrt\frac{d}{n}\right)$. On the other hand, if the network is from a random graph, the rate is $\frac{d}{\sqrt{n}}$ when $d\ll n^{1/2}$. 

While our positive result for random graphs is only true for $d\ll \sqrt{n}$, we conjecture that $d\gtrsim \sqrt{n}$ is the threshold for inconsistency, and that $\hat\rho -\rho$ converges to a possibly non-normal distribution. This would be analogous to what occurs in the weak instruments literature; for example, the weak IV framework of \citet{staiger1994instrumental} provides a signal strength below which the asymptotic distribution of the IV estimator is a mixture of normals. 

The 2SLS estimator has a higher variance for random, possibly non-bipartite graphs than the OLS estimator. This reflects the general phenomenon that IV estimators exhibit larger variances than their OLS counterparts, in spite of achieving consistent estimation in the presence of endogeneity.
\begin{proof}
    We introduce the same notation as in the OLS proof, $e(A,B)=V^TA^TMBV$ and $v(A,B)=V^TA^TMB\varepsilon$. In this notation, we have
\[\tilde{Z}_2^T\tilde{W}_2=
\begin{pmatrix}
    e(G,G)&e(G,GH)+v(G,GS^{-1})\\
    e(G^2,G)&e(G^2,GH)+v(G^2,GS^{-1})
\end{pmatrix}.
\]
Denoting $\phi:=\det(\tilde{Z}_2^T\tilde{W}_2)$, we write the inverse,
\[(\tilde{Z}_2^T\tilde{W}_2)^{-1}=
\frac{1}{\phi}\begin{pmatrix}
    e(G^2,GH)+v(G^2,GS^{-1})&-(e(G,GH)+v(G,GS^{-1}))\\
    -e(G^2,G)&e(G,G)
\end{pmatrix}.
\]
Assumption \ref{assumption:triangles} yields the following useful result:
\begin{lemma}\label{lemma:triangle-order}
    Under Assumption \ref{assumption:triangles}, $\Tr(G^TG^2)=o(\|G^2\|_F^2)$.
\end{lemma}
\begin{proof}
We have $\Tr(G^TG^2)\lesssim\frac{|C_3|}{d^3}\ll \frac{|C_4|}{d^4}\lesssim \Tr((G^2)^TG^2)=\|G^2\|_F^2$.
\end{proof}
Because $\|G\|_F$ and $\|G^2\|_F$ may grow at different rates, we introduce a rescaling matrix $F=\mathrm{diag}(\frac{1}{\|G\|_F},\frac{1}{\|G^2\|_F})$. As a consequence of our identification assumption, we have the following convergence result for the appropriately rescaled version of $\tilde{Z}_2^T\tilde{W}_2$:
\begin{prop}\label{prop:2sls-covariance}
$\|G^2\|_F(\tilde{Z}_2^T\tilde{W}_2)^{-1}F^{-1}\rightarrow_P \tilde\Gamma_{ZW}^{-1}$, where $\tilde\Gamma_{ZW}^{-1}$ is defined in Equation (\ref{eq:covariance-2sls}).
\end{prop}
We also have the following CLT for a rescaled version of $Z$, which follows from the same reasoning as for OLS:
\begin{prop}\label{prop:2sls-clt}
    For $F=\mathrm{diag}(\frac{1}{\|G\|_F},\frac{1}{\|G^2\|_F})$,
    \begin{align*}
        FZ^T\varepsilon =F\begin{pmatrix}
    v(G,I)\\ v(G^2,I)
\end{pmatrix}\rightarrow_D N(0, \tilde\Sigma_{ZZ}).
    \end{align*}
\end{prop}
Combining these via Slutsky's theorem, we have
\begin{align*}
\|G^2\|_F(\tilde{Z}_2^T\tilde{W}_2)^{-1}\tilde{Z}_2^T\varepsilon&=\|G^2\|_F(\tilde{Z}_2^T\tilde{W}_2)^{-1}F^{-1}F\tilde{Z}_2^T\varepsilon\rightarrow_D N(0,\tilde\Gamma_{ZW}^{-1}\tilde\Sigma_{ZZ}\tilde\Gamma_{ZW}^{-T}).
\end{align*}
\end{proof}
\section{Identifiability in Linear-in-Sums Models}\label{sec:wi-lis}
In the previous section, we showed that estimation of the linear-in-means model is subject to slow convergence or inconsistency when the network degree is increasing. In this section, we introduce the linear-in-sums model as an alternative and argue that it can usually be estimated at a standard rate via two-stage least squares.
\subsection{The Linear-in-Sums Model}\label{section:lis}
In the linear-in-means model, the estimation rate is slow because of the asymptotic collinearity between $1$ and regressors of the form $G^kX$ or $GY$. In this section, we show that the \textit{linear-in-sums (LIS)} model, defined by replacing the row-normalized adjacency matrix $G$ with the original adjacency matrix $A$, generally does not suffer from the same issues as long as there is sufficient variation in the graph structure.

We first outline the baseline analysis for the LIS model, ignoring weak identification concerns. The model is defined as
$$Y=\alpha 1 + X\beta + AX\delta_n + \rho_n AY + \varepsilon.$$
We emphasize that $\delta_n$ and $\rho_n$ are parameter sequences depending on $n$, which is necessary to make $(I-\rho_n A)^{-1}$ well-defined. Henceforth we remove the subscript for ease of notation. We maintain the same assumptions as with the LIM model, including near-degree regularity, i.e., the degrees $d_i$ are uniformly in the interval $(c_1d,c_2d)$ for some constants $c_1,c_2$. To ensure that the model is well-defined and to keep $Y$ on a constant scale, we enforce that $|\rho|<1/\lambda_1(A)$, which implies that the inverse $(I-\rho A)^{-1}$ exists and can be written via the Neumann series, $(I-\rho A)^{-1}=\sum_{k=0}^\infty (\rho A)^k$. We further assume that $d\rho \rightarrow_n \rho_0$ for some constant $\rho_0$; since $\lambda_1(A)=\Theta(d)$ under our near-degree regularity assumption, this implies $\rho$ is on the same scale as $1/\lambda_1(A)$. We impose an analogous restriction on $\delta$: $d\delta \rightarrow_n \delta_0$ for some constant $\delta_0$. We do not need an extra boundedness condition on $\delta$ since the contextual effect $\delta$ does not show up in the inverted term, $(I-\rho A)^{-1}$.

Since the regressors are on different scales, the baseline convergence rate is different for the four parameters, as $A^kX$ typically scales as $O(d^k)$. To analyze this, we thus consider rescaling matrices $F=\mathrm{diag}(1,1,1/d, 1/d^2)$ and $H=\mathrm{diag}(1,1,1/d, 1/d)$ for the instruments $Z$ and endogenous $W$ respectively. For the 2SLS estimator, we have
\begin{align*}
    H^{-1}(\hat\theta_{2SLS}-\theta) &= H^{-1}(W^TP_ZW)^{-1}(P_ZW)^T\varepsilon\\
    &=H^{-1}(W^TZ(Z^TZ)^{-1}Z^TW)^{-1}W^TZ(Z^TZ)^{-1}Z^T\varepsilon\\
    &=(HW^TZF(FZ^TZF)^{-1}FZ^TWH)^{-1}HW^TZF(FZ^TZF)^{-1}FZ^T\varepsilon
\end{align*}
Along with the same conditions on $\varepsilon$ imposed in the previous section, the standard 2SLS assumptions are now imposed on the rescaled covariates, $ZF$ and $WH$:
\begin{assumption}\label{assumption:lis-moments}
    $\frac{1}{n}(WH)^T(ZF)\rightarrow_P \Gamma_{WZ}$ and $\frac{1}{n}(ZF)^TZF\rightarrow_P \Gamma_{ZZ}$, where $\Gamma_{WZ}$ and $\Gamma_{ZZ}$ are full-rank.
\end{assumption}
Under these assumptions, standard arguments yield the following:
\begin{prop}
    $\sqrt{n}H^{-1}(\hat\theta_{2SLS}-\theta)\rightarrow N(0,(\Gamma_{WZ}\Gamma_{ZZ}^{-1}\Gamma_{ZW})^{-1})$.
\end{prop}
That is, the network components of $\hat\theta$ converge at the rate $O_P\left(\frac{1}{d\sqrt{n}}\right)$ - equivalently, the rate for the estimating the $d$-rescaled coefficients is standard root-n.

We can provide conditions under which this invertibility assumption holds and thus the collinearity issues for linear-in-means are avoided. First, since $X$ is not responsible for collinearity as long as its variance is positive, we only need to consider collinearity among the intercept and network features (this is equivalent to partialling out $X$). We first examine $\frac{1}{n}F(Z^TZ)F$:
\begin{align*}
    \frac{1}{n}F(Z^TZ)F=\begin{pmatrix}
        1&\frac{1}{nd}1^TAX&\frac{1}{nd^2}1^TA^2X\\
        \cdot&\frac{1}{nd^2}X^TA^2X&\frac{1}{nd^3}X^TA^3X\\
        \cdot&\cdot&\frac{1}{nd^4}X^TA^4X
    \end{pmatrix}.
\end{align*}
To provide a broad class of networks where we can show convergence and invertibility of the matrix, we consider $A$ sampled from a graphon, which is commonly used to model dense networks \citep{bickel2011method}:
\begin{align*}
    A_{ij}&=Bern(p_n f(U_i,U_j)),\\
    U_i&\sim U(0,1),
\end{align*}
where $U_i$ are i.i.d., $A$ is symmetric, and $f:[0,1]\times [0,1]\mapsto \mathbb{R}_{\geq 0}$ is a symmetric, integrable function known as a \textit{graphon}. In this setting, $p_n$ governs the sparsity of the graph, and we can take $d=np_n$. The following convergence result is adapted fro \citet{avella2018centrality}:
\begin{prop}
    If $f$ is piecewise Lipschitz and the degree grows at any polynomial rate $n^{\alpha}$ for $\alpha\in(0,1]$, then $\frac{1}{nd^k}1^TA^k1\rightarrow \int_{u_1,...,u_{k+1}\in (0,1)} \prod_{i=1}^k f(u_i,u_{i+1})du_1...du_{k+1}=:m_k$.
\end{prop}
Consequently, $\left(\frac{1}{n}F(Z^TZ)F\right)_{ij}\rightarrow m_{i+j-2}$. In the language of graph limits, $m_k$ is the homomorphism density of the length-$(k{-}1)$ path in the graphon $f$. With convergence of moments settled, we can derive conditions for invertibility of the limiting moment matrix. In the following, we provide two conditions, the first of which only uses integrals of the graphon $f$, and the second of which is based on the spectral decomposition of the graphon.
\subsection{Identification from Graphon Degree and Codegree}
We assume here that the instruments are $1,AX,A^2X$. Define the \textit{degree} function $g_1(u)=\int_v f(u,v)$, and the \textit{codegree} function $g_2(u)=\int_{v,w} f(u,v)f(v,w)$. Then the moment matrix has the following structure:
\[
\Gamma_{ZZ}=\begin{pmatrix}
        1&\int g_1&\int g_1g_2\\
        \cdot&\int g_1^2&\int g_1^2g_2\\
        \cdot&\cdot&\int g_2^2
    \end{pmatrix}.
\]
The decomposition can be verified by re-arranging the integrals: for example, the bottom right entry is
\begin{align*}
&\int_{u_1,...,u_5}f(u_1,u_2)f(u_2,u_3)f(u_3,u_4)f(u_4,u_5)\\={}&\int_{u_3}\int_{u_1,u_2}f(u_1,u_2)f(u_2,u_3)\int_{u_4,u_5}f(u_3,u_4)f(u_4,u_5)\\
={}&\int_{u}g_2(u)^2.
\end{align*}
From this representation, the following necessary and sufficient condition is immediate:
\begin{theorem}
    $\Gamma_{ZZ}$ is invertible iff $1,g_1,g_2$ are linearly independent.
\end{theorem}
The definitions of $g_1$ and $g_2$ are related to the \textit{link} and \textit{codegree} functions used by \citet{auerbach2022identification} to identify a partially linear model with graphon-controlled semiparametric components. Our $g_1$ and $g_2$ are the integrals of the agent link and codegree functions as defined in Auerbach's work. The intuition is that variation in the expected degrees of each node is required for $1$ and $g_1$ to be linearly independent; furthermore, there needs to be excess variation in the number of two-hop neighbors to achieve identification. The condition is also similar to the standard assumption usually employed in linear-in-means models for identification, namely that $I,G,G^2$ are linearly independent. However, in the growing network setting, the corresponding identification assumption for linear-in-means does not imply that the limiting $\Gamma_{ZZ}$ is invertible, even if we impose the graphon assumption. This is because the row-normalization forces the constant vector to be an eigenvector of $G$, so that $GX$ and $G^2X$ limit to constant vectors when $X$ is IID.
\subsection{Identification from Spectral Variation}\label{sec:spectral}
By appealing to the spectral decomposition of a graphon, we can provide alternative conditions for identification which are natural for low-rank network models, such as the stochastic block model \citep{lee2019review}. Following the discussion from \citet{avella2018centrality}, we define the graphon operator:
\begin{definition}
    For a graphon $f$, the associated graphon operator $T_f:L^2([0,1])\rightarrow L^2([0,1])$ is defined as
    $$(T_fg)(u)=\int_v f(u,v)g(v)dv.$$
\end{definition}
Graphon operators admit a spectral decomposition:
\begin{prop}
    The graphon operator admits a decomposition
    $$T_fg=\sum_{i=1}^\infty \lambda_i \langle \phi_i,g\rangle \phi_i,$$
    where $\phi_i$ form an orthonormal basis in $L^2$ and are eigenfunctions with corresponding eigenvalues $\lambda_i$.
\end{prop}
With this decomposition in mind, we can write
$$m_k=\int_u (T_f^k1)(u)du=\left\langle 1,\sum_{i=1}^\infty \lambda_i^k\langle \phi_i,1\rangle \phi_i\right\rangle=\sum_{i=1}^\infty \langle \phi_i,1\rangle^2\lambda_i^k.$$
Consequently, we can write the moment matrix in the form $\sum_{i=1}^\infty a_i^2b_ib_i^T$, where $a_i=\langle \phi_i,1\rangle$ and $b_i=(1,\lambda_i,\lambda_i^2)$. We can see that this matrix is full rank if and only if there are three linearly independent $b_i$ corresponding to $a_i\neq 0$. Using the Vandermonde structure of $M$ \citep{horn2012matrix}, it follows that any set of three $b_i$'s are linearly independent if and only if the eigenvalues $\lambda_i$ are distinct. Combining these observations yields the following identification condition:
\begin{theorem}\label{thm:lis1}
    $\Gamma_{ZZ}$ is invertible if and only if the graphon operator $T_f$ has at least three distinct eigenvalues corresponding to eigenfunctions not orthogonal to $1$.
\end{theorem} 
The theorem generalizes to the case where multiple instruments are used: to distinguish between $k+1$ values $m_0,...,m_k$, we now require $k+1$ distinct eigenvalues. 

In the following section, we provide an example of identification in stochastic block models.
\subsection{Identification in Stochastic Block Models}
The stochastic block model is a common model of communities in social networks \citep{lee2019review}. The model is parameterized by a matrix $P\in [0,1]^{K\times K}$ and vector $\pi\in [0,1]^K$, where $K$ is the number of communities, $P_{ij}$ is the probability of connection between individuals in communities $i$ and $j$, and $\pi_i$ is the probability an individual is in community $i$. For each individual, a single community membership is drawn independently, and conditioned on the memberships, edges are sampled independently of one another. 

We further define $E_{SBM}=PQ$, where $Q=\mathrm{diag}(\pi)$. The spectral properties of the graphon $f_{SBM}$ can be derived from those of the associated SBM, as shown in the following proposition:
\begin{prop}[Lemma 2 from \citet{avella2018centrality}]
    The eigenvalues of $T_{f_{SBM}}$ are in one-to-one correspondence with the eigenvalues of $E_{SBM}$.
\end{prop}
Thus for the SBM, we have identifiability if and only if $E_{SBM}$ has at least 3 distinct eigenvalues corresponding to eigenvectors not orthogonal to $1$. We can expect this to be true for most values of $E_{SBM}$ when the number of communities, $K$, is at least $3$. Identification can fail because there are fewer than 3 distinct eigenvalues, but it can also fail because the number of eigenvectors which have nonzero inner product with the constant vector, $1$, is too small. For example, consider a disconnected community matrix,
\[
E_{SBM}=\frac{1}{3}
\begin{pmatrix}
    1&0&0\\
    0&1&\frac{1}{2}\\
    0&\frac{1}{2}&1
\end{pmatrix}.
\]
The eigenvalues, $1/3$, $5/6$, and $1/6$, are distinct, and $1$ is not an eigenvector, but the last eigenvector $(0,1,-1)^T$ is orthogonal to 1. It turns out that even by enforcing strict positivity of entries and positive-definiteness of $E_{SBM}$, one can numerically construct values of $E_{SBM}$ that satisfy the distinct eigenvalue condition but not the orthogonality condition, such as the following matrix for which the eigenvalues are distinct, but the bottom eigenvector is orthogonal to $1$:
\[
\begin{pmatrix}
0.2321 & 0.0718 & 0.0295 \\
0.0718 & 0.0728 & 0.0287 \\
0.0295 & 0.0287 & 0.0618
\end{pmatrix}
\]
However, violation of the identifiability condition is still an edge case, as it is generally restored after a small perturbation of $E_{SBM}$.
\subsection{Instrument Relevance Condition}
Aside from the condition that $\Gamma_{ZZ}$ is full-rank, we also need the condition that $\Gamma_{ZW}$ is full-rank. This is the instrument relevance condition, namely, that the instrument $A^2X$ is correlated with the endogenous $AY$. We argue that the conditions above are sufficient. First, we note that from calculations involving the sampled graphon (i.e., Theorem 2 from \citet{avella2018centrality}), moments of the form $X^TA^mAY$ behave as
\begin{align*}
    \frac{1}{nd^{m+1}}X^TA^mAY&\approx \frac{1}{nd^{m+1}}\mu1^TA^{m+1}(I-\rho A)^{-1}(\alpha 1+\mu\beta 1+\mu\delta A1 + \varepsilon)\\
    &\rightarrow \sum_{k=1}^\infty a_k\mu\lambda_k^{m+1}(\alpha+\mu(\beta+\delta_0\lambda_k))\frac{1}{1-\rho_0 \lambda_k}\\
    &=: \mu \sum_{k=1}^\infty a_k\lambda_k^mh(\lambda_k),
\end{align*}
where $h(\lambda)=\frac{1}{1-\rho_0 \lambda}\lambda(\alpha+\mu\beta+\mu\delta_0\lambda)$, and $a_k$ and $\lambda_k$ are as defined in Section \ref{sec:spectral}. Thus with the rescaling matrix definitions of $F$ and $H$ provided earlier, we have
\begin{align*}
\Gamma_{ZW}= \begin{pmatrix}
\sum_ia_i^2&\sum_ia_i^2\lambda_i&\sum_ia_i^2h(\lambda_i)\\
\sum_ia_i^2\lambda_i &\sum_ia_i^2\lambda_i^2&\sum_ia_i^2\lambda_ih(\lambda_i)\\
\sum_ia_i^2\lambda_i^2&\sum_ia_i^2\lambda_i^3&\sum_ia_i^2\lambda_i^2h(\lambda_i)
\end{pmatrix}.
\end{align*}
This limiting matrix can be written as $BC^T$, where $B$ has $b_k:=(1,\lambda_k, h(\lambda_k))$ as columns and $C$ has $(1,\lambda_k,\lambda_k^2)$ as columns. Under the conditions in Theorem \ref{thm:lis1}, $C$ is full rank. Subsequently, identification depends on the rank of $B$. We now show that the same conditions generally suffice for invertibility of $\Gamma_{ZW}$, and so under the conditions of Theorem \ref{thm:lis1}, the 2SLS estimator of the LIS model converges at a standard rate.
\begin{theorem}
Under the same conditions as Theorem \ref{thm:lis1}, if $(\alpha+\mu\beta)\rho_0+\mu\delta_0\neq 0$, then $\Gamma_{ZW}$ is invertible, and so two-stage least squares is consistent at a standard rate.
\end{theorem}
\begin{proof}
    It suffices to show that if there are three distinct eigenvalues $\lambda_i$ with corresponding $a_i^2\neq 0$, then the three vectors of the form $(1,\lambda_i, h(\lambda_i))$ are linearly independent. Suppose the contrary. Then there exist $\alpha, \beta$ so that $\alpha+\beta\lambda_i=h(\lambda_i)$ for each $i$. This implies $\frac{h(\lambda_i)-h(\lambda_j)}{\lambda_i-\lambda_j}=\beta$ for each pair of the three $\lambda_i$. That is, $(\lambda_1,h(\lambda_1))$, $(\lambda_2,h(\lambda_2))$, and $(\lambda_3,h(\lambda_3))$ are all distinct points on the same line in $\mathbb{R}^2$. We will argue that $h$ is strictly convex or strictly concave on its domain, $(-\frac{1}{\rho_0},\frac{1}{\rho_0}).$ Since any line can intersect the graph of a strictly convex or strictly concave function at most twice, this yields the contradiction.

    By taking the second derivative of $h$, we arrive at
    $$h''(\lambda)=\frac{2((\alpha + \mu\beta)\rho_0 + \mu\delta_0)}{(1-\rho_0 \lambda)^3}.$$
    Since $(1-\rho_0 \lambda)^3$ is positive over the domain and the numerator is non-zero, $h$ is strictly convex or strictly concave on its domain.
\end{proof}
\section{Simulations}\label{sec:wi-simulations}
In this section, we validate our theory through a variety of numerical experiments.
\subsection{Random Graph}
We compare the behaviors of OLS and 2SLS for an Erd\H{o}s--R\'enyi graph. For each $n=100,200,...,2000$, we sample an Erd\H{o}s--R\'enyi random graph with expected degree $d$, where $d=O(n^{1/2}),O(n^{1/4})$. The constant is chosen so that for $n=100$, the expected degree is $10$. We fix $\alpha=1$, $\beta=1.5$, $\delta=0.6$, $\rho=0.3$. The $X_i$ are drawn from a normal distribution with mean $2$ and variance $1$, and the $\varepsilon_i$ are drawn from a normal distribution with standard deviation $0.1$. The resulting values of $\hat{\delta}-\delta$ and $\hat{\rho}-\rho$, along with 95 percent pointwise confidence bands, are plotted in Figures \ref{fig:er1} and \ref{fig:er2}.
\begin{figure}[t]
    \centering
    \includegraphics[trim={0 0 0 1.5cm}, clip, scale=0.3]{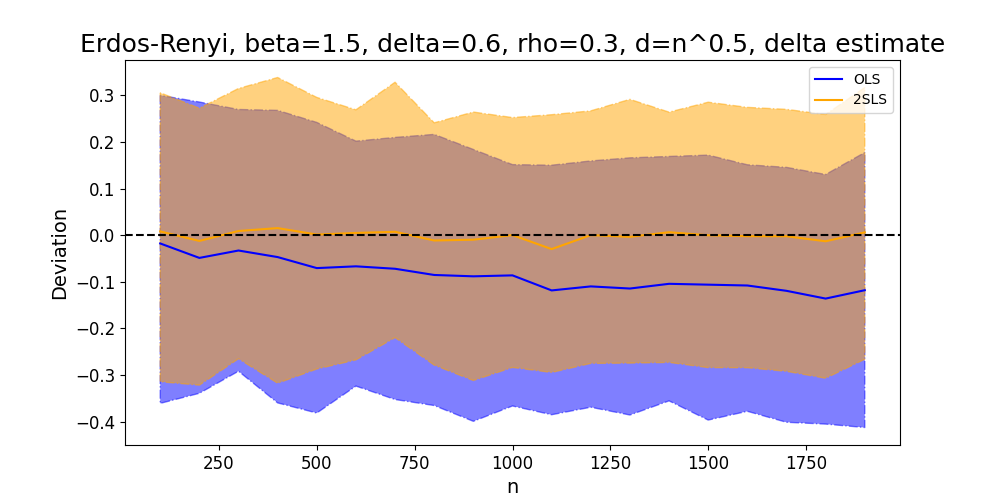}
    \includegraphics[trim={0 0 0 1.5cm}, clip, scale=0.3]{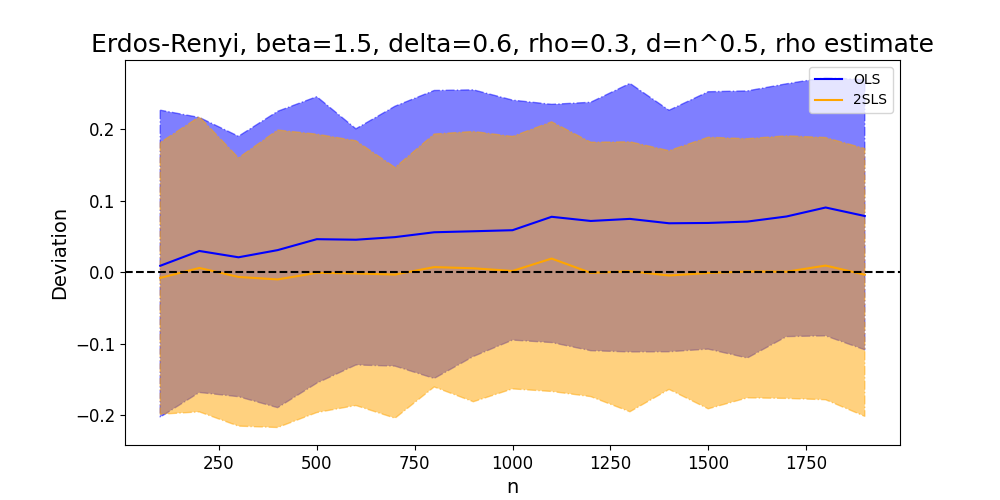}
    \caption{Error of OLS and 2SLS estimators of $\delta$ (left) and $\rho$ (right) for Erd\H{o}s--R\'enyi graph with average degree $d=n^{1/2}$. The mean errors for the estimators are highlighed in solid lines.}
    \label{fig:er1}
\end{figure}
\begin{figure}[t]
    \centering
    \includegraphics[trim={0 0 0 1.5cm}, clip, scale=0.3]{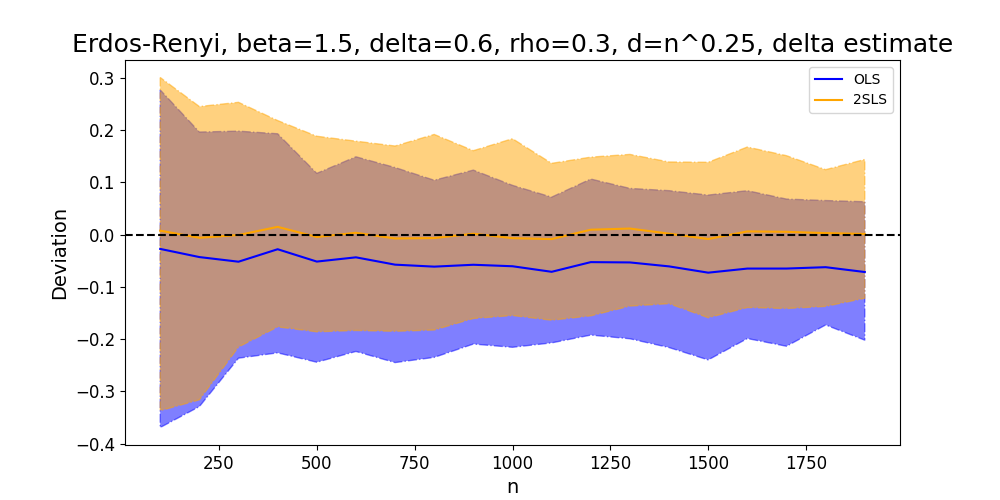}
    \includegraphics[trim={0 0 0 1.5cm}, clip, scale=0.3]{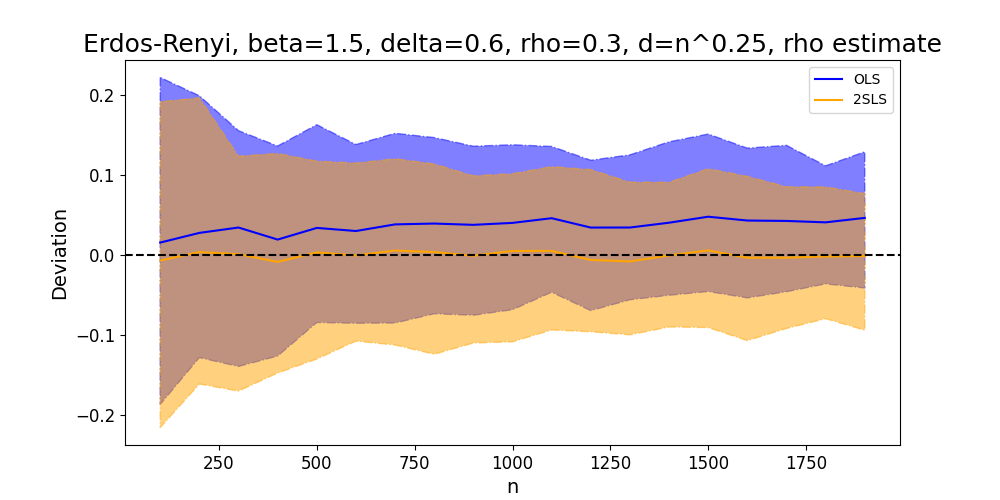}
    \caption{Error of OLS and 2SLS estimators of $\delta$ (left) and $\rho$ (right) for Erd\H{o}s--R\'enyi graph with average degree $d=n^{1/4}$.}
    \label{fig:er2}
\end{figure}
From these figures, it is evident that there is a nonzero bias in the OLS estimates for $\delta$ and $\rho$, as predicted by Theorem \ref{thm:ols-formal}. The sign of the bias is predictable from the expressions we derived, i.e., it is negative for the contextual effect and positive for the endogenous effect when $\rho>0$. We can see in Figure \ref{fig:er2} that the estimator error is converging to the bias. We expect this to be true as well in Figure \ref{fig:er1} for the denser graph $d=n^{1/2}$, but we believe the sample size is too small to see the convergence, since the bias is still increasing significantly at $n=2000$.

In contrast, the 2SLS estimator does not have bias, i.e., the orange line meets the dashed black line at zero. For $d=n^{0.5}$, the estimate for $\rho$ appears to maintain a constant amount of variation. This agrees with our analysis, which suggested that for random graphs, $d\gtrsim \sqrt{n}$ is the threshold where $\hat{\rho}_{2SLS}$ is inconsistent and has a non-degenerate limiting distribution. For moderate values of density $d=n^{1/4}$, we can see that the estimates from 2SLS converge at a rate slower than $\sqrt{n}$.
\subsection{OLS when $\beta=0$}
In the discussion after Theorem \ref{thm:ols-formal}, we suggested a surprising result that when $\beta=0$, OLS could be asymptotically unbiased for $\delta$ if the graph structure is not highly clustered. In Figure~\ref{fig:er_zero}, we validate that this is indeed the case: for random graphs, as we described, OLS appears to be asymptotically unbiased for $\delta$ and closely matches the behavior of 2SLS. However, the estimate for $\rho$ still admits a non-zero bias.
\begin{figure}[t]
    \centering
    \includegraphics[trim={0 0 0 1.5cm}, clip, scale=0.3]{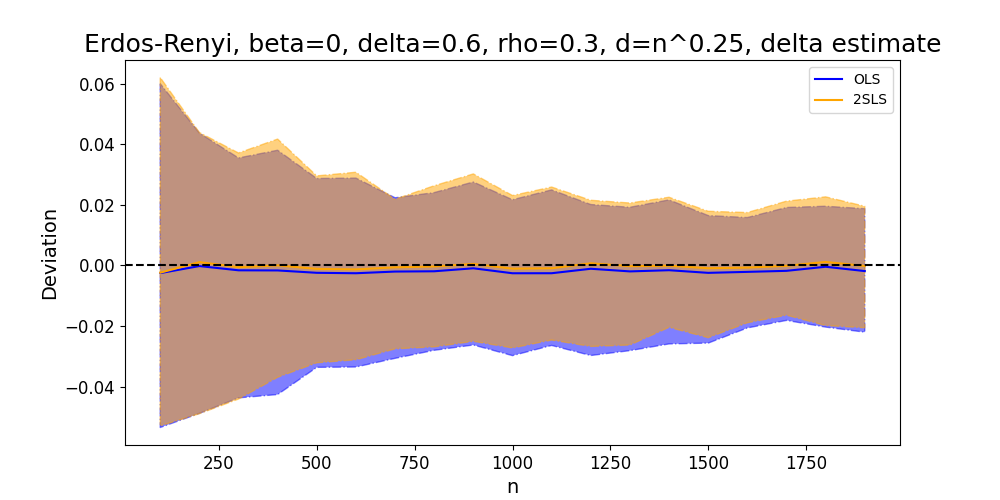}
    \includegraphics[trim={0 0 0 1.5cm}, clip, scale=0.3]{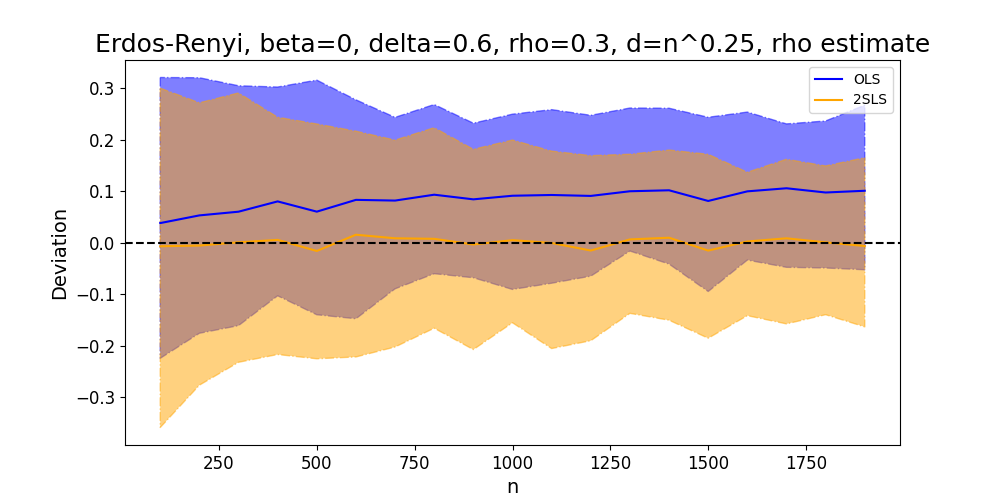}
    \caption{Error of OLS and 2SLS estimators of $\delta$ (left) and $\rho$ (right) for Erd\H{o}s--R\'enyi graph with average degree $d=n^{1/4}$ when $\beta=0$.}
    \label{fig:er_zero}
\end{figure}
\subsection{Complete Bipartite Graphs}
For this experiment, we generated a disjoint union of $d$-regular complete bipartite graphs, with $n$ nodes total. This union of bipartite cliques has different properties from random graphs: for random graphs, $\Tr(A^3)=\Theta(d^3)$ and $\Tr(A^4)=\Theta(d^4)$, while for this ensemble, $\Tr(A^3)=0$ and $\Tr(A^4)=\Theta(nd^3)$, i.e., it has a minimal number of triangles and maximal number of 4-cycles. Our theory suggests that when $\beta=0$, the OLS estimate of $\delta$ is still biased, while the 2SLS estimate of $\rho$ enjoys an improved convergence rate of $O\left(\sqrt\frac{d}{n}\right)$.

\begin{figure}
    \centering
    \includegraphics[trim={0 0 0 1.5cm}, clip, scale=0.3]{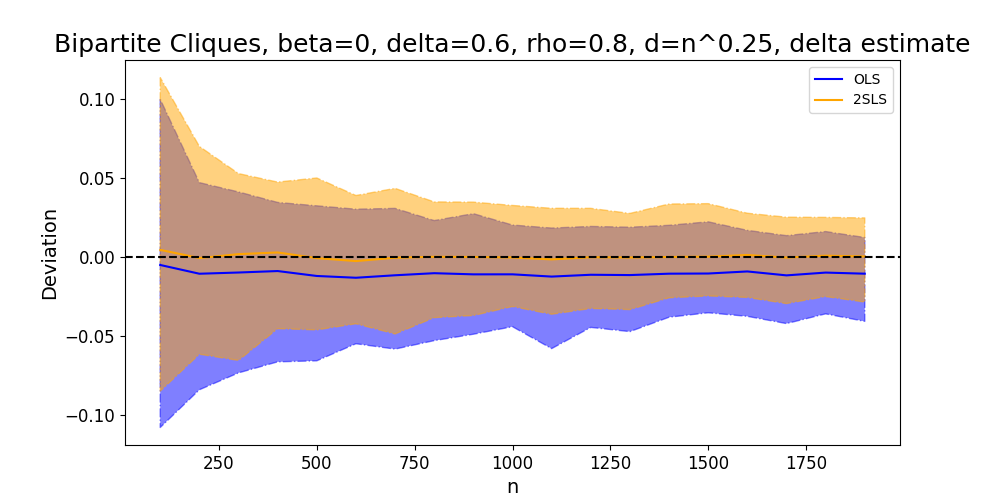}
    \includegraphics[trim={0 0 0 1.5cm}, clip, scale=0.3]{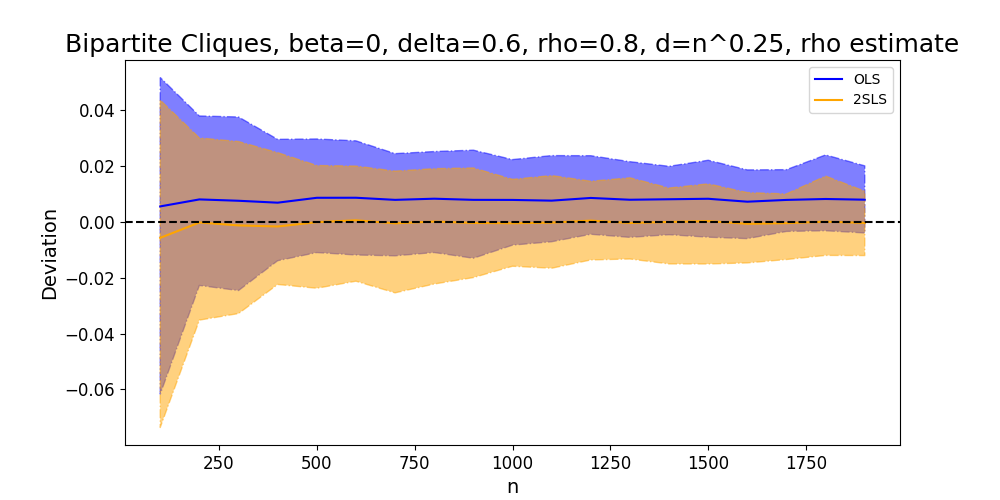}
    \caption{Error of OLS and 2SLS estimators of $\delta$ (left) and $\rho$ (right) for the union of complete bipartite graphs with average degree $d=n^{1/4}$ when $\beta=0$.}
    \label{fig:bipartite-clique}
\end{figure}
Figure \ref{fig:bipartite-clique} demonstrates that this is indeed the case. Although OLS estimation of $\delta$ was asymptotically unbiased for random graphs when $\beta=0$, for the union of complete bipartite graphs, there is a non-vanishing bias term. On the other hand, $\hat{\rho}_{2SLS}$ enjoys an improved convergence rate, as compared visually with Figure \ref{fig:er_zero}. This example highlights the complexities of the linear-in-means model: the parameter values and graph structure can affect OLS and 2SLS estimators differently.
\subsection{Cliques}
Finally, we consider the case where the network is a disjoint union of $d$-cliques. In this case, our theory for OLS estimation still holds, and we expect that the estimator will converge to a biased solution. However, for the 2SLS theory, we prohibited such networks in our identifiability assumption. In Figures \ref{fig:clique-ols} and \ref{fig:clique-2sls}, we plot the results of estimating the linear-in-means model in this regime. We can see that the OLS estimates, while biased, are stable. However, the 2SLS estimates are wildly unstable. This is to be expected: it is known that Manski's reflection problem can be solved from an identification perspective by excluding oneself from their own neighborhood for group averaging \citep{bramoulle2020peer}, but standard identification does not imply strong identification in our asymptotic framework. The weak instrument problems thus lead to highly unreliable estimates.
\begin{figure}[t]
    \centering
    \includegraphics[trim={0 0 0 1.5cm}, clip, scale=0.3]{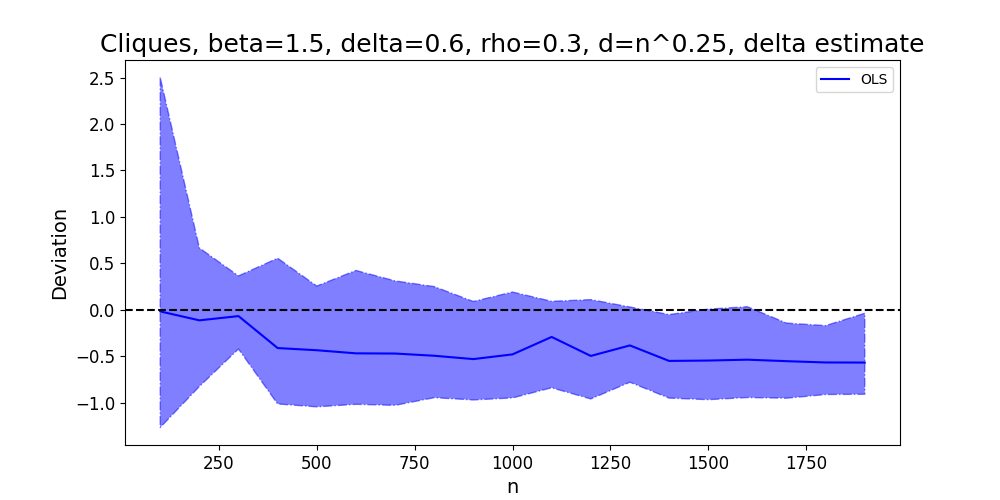}
    \includegraphics[trim={0 0 0 1.5cm}, clip, scale=0.3]{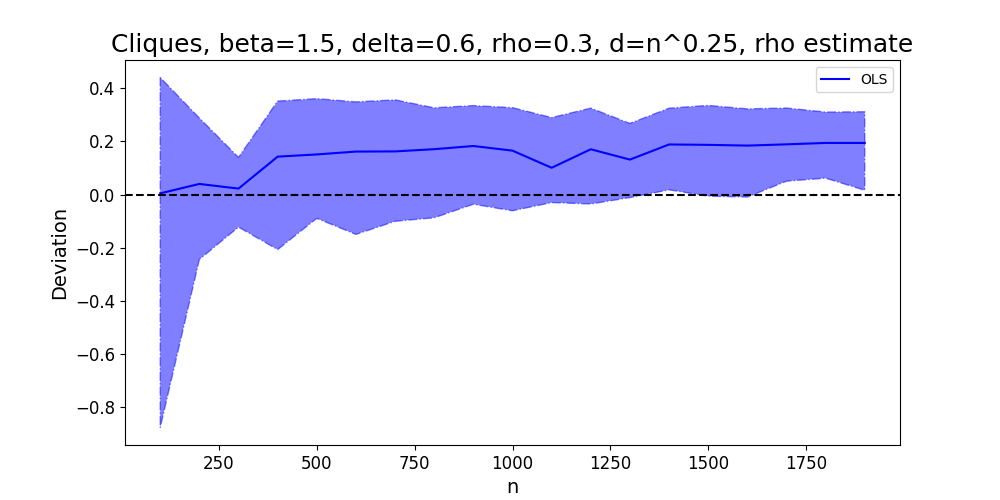}
    \caption{Error of OLS estimator of $\delta$ (left) and $\rho$ (right) for the union of cliques with average degree $d=n^{1/4}$.}
    \label{fig:clique-ols}
\end{figure}
\begin{figure}
    \centering
    \includegraphics[trim={0 0 0 1.5cm}, clip, scale=0.3]{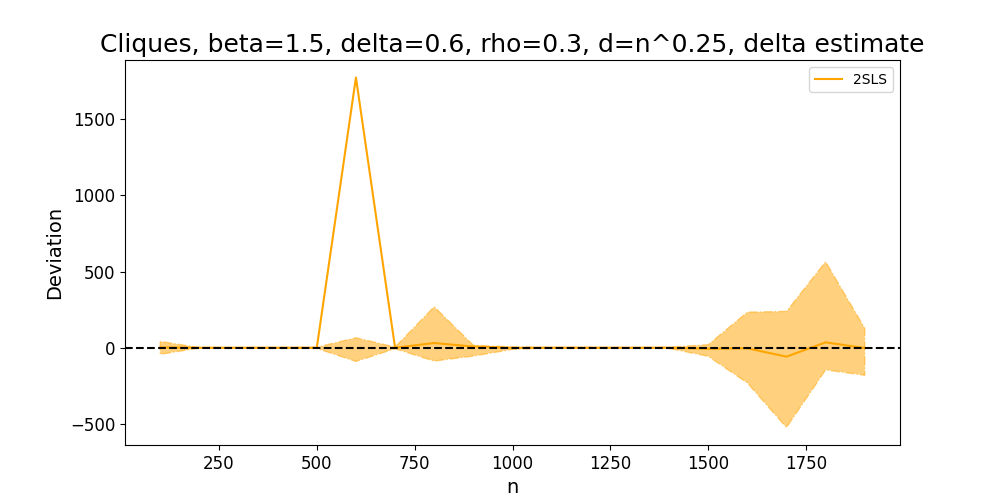}
    \includegraphics[trim={0 0 0 1.5cm}, clip, scale=0.3]{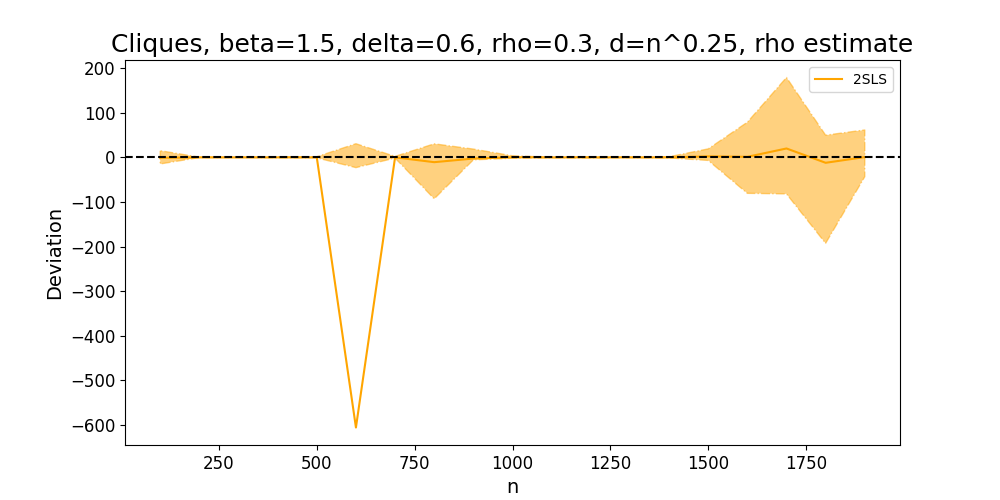}
    \caption{Error of 2SLS estimator of $\delta$ (left) and $\rho$ (right) for the union of cliques with average degree $d=n^{1/4}$.}
    \label{fig:clique-2sls}
\end{figure}

\section{Conclusion}\label{sec:wi-conclusion}
In this work, we investigated the phenomenon of weak identification in the popular linear-in-means model of peer effects. By concentrating our analysis on the network contextual and endogenous peer effects, we showed that the behavior of OLS and 2SLS estimators can be complex when the degrees of the network are allowed to grow: the model parameters and network structure can contribute jointly to produce biases or slower rates of convergence. We also showed that this weak identification is generally not a problem in the linear-in-sums model: for a large class of random graphs, the linear-in-sums model is strongly identified. Our results indicate that the linear-in-sums model can be efficiently estimated in settings where linear-in-means models suffer from slow rates.

One broader theme that is reflected in this work is the difference between OLS and instrumental variable estimation methods. The weak identification literature has provided diagnostics to determine whether or not an instrument can reliably used in IV estimation~\citep{staiger1994instrumental}. For the linear-in-means model, while instrumental variables generally outperforms ordinary least-squares, when there truly is weak identification, as in the case of the union of cliques, instrumental variables estimation is unreliable. In such a setting, OLS is biased, but can still potentially be used to say something informative about the parameters of interest if the bias is small or known to be in a particular direction.

There are a few interesting possible extensions to this work. The asymptotic framework we studied corresponds to the \textit{infill asymptotic} setting, where a single network is observed. In contrast, some applications fit in the \textit{increasing-domain} asymptotic framework, where many isolated communities, such as villages, are observed. In such cases, it is common to incorporate village-level fixed effects into the model. This creates an incidental parameter problem that is often resolved by differencing. Since this procedure removes useful information from the data, a stronger condition based not only on triangles, but also on fourth-order motifs, is required to ensure identification \citep{bramoulle2009identification}. Once this differencing is performed, identification has been shown to be impossible in both complete graphs and complete bipartite graphs. Similarly, we would expect that the identification conditions for 2SLS in our setup would need to change in order to limit the number of triangles and 4-cycles; it is likely that the estimation rate would then depend on quantities such as $\|G^3\|_F$ instead. We leave this extension for future work.

Additionally, it would be interesting to see how the results extend to networks with degree heterogeneity, as discussed in the related work. Degree heterogeneity is known to be widespread but problematic in other areas of statistical network analysis, such as community detection, as it slows down concentration of useful random matrix quantities \citep{ke2024optimal}. However, this could potentially be beneficial for the weak instrument problem, since rapid averaging, or concentration, is the cause of weak instruments in our setup. Nevertheless, degree heterogeneity poses other problems for inference: it may be the case that the limiting covariance matrix depends on the values of covariates for specific individuals due to concentration of influence \citep{acemoglu2012network,halpern2021defense}. This would make inference non-standard when interpreting regressors as random, as is commonly assumed in econometrics and statistics \citep{abadie2020sampling,buja2019models}. Studying the impact of degree heterogeneity seems to be a complex but worthwhile endeavor that can help us understand how the theory extends to more practically relevant network models.

Another question is whether or not the estimation rates derived are fundamental lower bounds over the class of all estimators. In particular, we can ask if there are estimators which dominate the two-stage least squares estimator. It is known that in the strongly-identified case, there is a three-stage procedure which is efficient among estimators using $GX,G^2X,...$ as instruments \citep{lee2003best,zhang2010more}. How do these estimators perform in weakly identified setups? Additionally, it may be possible for regularized estimators to outperform their unregularized counterparts in this setup, as suggested in some related work on weakly identified models~\citep{knight2008shrinkage,tchuente2019weak}.

Finally, it would be worthwhile to explore further connections with the weak instruments literature. A large body of work has been devoted to developing finite-sample approximations of estimator distributions \citep{staiger1994instrumental,bound1995problems,andrews2019weak} or devising robust testing procedures which have the correct size regardless of the strength of identification \citep{mikusheva2010robust}. In settings where peer effects are weakly identified, can these procedures be used to improve estimates? Developing improved techniques for inference could help increase the credibility of peer effects research.
\section{Acknowledgment}
This work was supported by ARO MURI W911 NF-19-1-0217 and a
Vannevar Bush Fellowship from the Office of the Under Secretary of
Defense.
\begin{appendix}
\section{Appendix}\label{appn} 
\subsection{Helpful Auxiliary Results}
Throughout the proofs, we will make use of the following observation, which is trivial by considering the symmetry of $A$ and the bound of $\Theta(\frac{1}{d})$ on the non-zero entries of $G$:
\begin{lemma}\label{lemma:order-equivalence}
    If $B$ is a product of $k$ terms of the form $G$ and $G^T$, and $G$ satisfies  Assumption \ref{assumption:network-reg}, then $\Tr(B)=\Theta_P(\Tr(G^k))$.
\end{lemma}
\begin{proof}
    Note that 
    \begin{align}
        \Tr(G^k)&=\sum_{c\in C_k}\prod_{(u,v)\in c}\frac{1}{d_u}=\Theta_P\left(\frac{|C_k|}{d^k}\right),
    \end{align}
    where $C_k$ is the set of all $k$-cycles and we have used Assumption \ref{assumption:network-reg} to bound the order of the reciprocal degrees. Since the sparsity pattern of $G^T$ is the same as that of $G$ by symmetry, it follows that $B$ is of the same order, $\Theta\left(\frac{|C_k|}{d^k}\right)$.
\end{proof}
\begin{lemma}\label{lemma:col-sums}
    For matrices $A,B$ with bounded row and column sums, $AB$ has bounded row and column sums. Additionally, $\|A\|_2$ is of constant order.
\end{lemma}
\begin{proof}
    The row sums are of the form $e_iAB1\leq c_1 e_iA1\leq c_2$ for some constants $c_1,c_2$, where $e_i$ is the $i$-th canonical basis vector, and similarly for column sums. The bound on the largest singular value is a consequence of the Gershgorin circle theorem applied to $A^TA$, which has bounded row and column sums.
\end{proof}
\subsection{Proofs for Section \ref{section:network-asymptotics}}
\begin{proof}[Proof of Proposition \ref{prop:rates}]
For $\|G\|_F^2$, we can compute $$\|G\|_F^2=\sum_{i=1}^n\sum_{j=1}^nG_{ji}^2=\sum_{j=1}^n\sum_{i\in N(j)}\frac{1}{d_j^2}=\sum_{j=1}^n\frac{1}{d_j}=\Theta\left(\frac{n}{d}\right).$$
For $\|G^2\|_F^2$, note that $\Tr((G^2)^TG^2)$ can be interpreted as a sum over 4-cycles in the graph. Since every cycle of the form $i\rightarrow j\rightarrow k$ contributes a 4-cycle and there are at least $\Omega(nd^2)$ of these motifs, this quantity is at least $\Omega\left(\frac{n}{d^2}\right)$.
\end{proof}
\begin{proof}[Proof of Proposition \ref{prop:low-rank}]
    The moment matrix is
    \begin{align}
        \frac{1}{n}Z^TW&=\begin{pmatrix}
            1&\frac{1}{n}1^TX&\frac{1}{n}1^TGX&\frac{1}{n}1^TGY\\
            \frac{1}{n}X^T1 &\frac{1}{n}X^TX&\frac{1}{n}X^TGX&\frac{1}{n}X^TGY\\
            \frac{1}{n}(GX)^T1 &\frac{1}{n}(GX)^TX&\frac{1}{n}(GX)^TGX&\frac{1}{n}(GX)^TGY\\
            \frac{1}{n}(G^2X)^T1&\frac{1}{n}(G^2X)^TX&\frac{1}{n}(G^2X)^TGX&\frac{1}{n}(G^2X)^TGY
        \end{pmatrix}.
    \end{align}
    The expectation of $\frac{1}{n}1^TB^TX$ is $\mu$ if $B1=1$, as is this case for $B=G$ and $G^2$. The expectation of $\frac{1}{n}X^TB^TX$ is $\sigma^2\frac{\Tr(B)}{n}+\mu\rightarrow \mu$ if $B1=1$ or $B^T1=1$ and $\Tr(B)\ll n$. Finally, using the notation $Y=\alpha S^{-1}1+HX+S^{-1}\varepsilon$, we can appropriately define $\mu'=E[1^TGY]$. Combining all of these, we can see that the only variance term which remains after taking limits is the variance term corresponding to $\frac{1}{n}X^TX$, and the mean signal term is rank 1 of the form $(1,\mu,\mu,\mu)^T(1,\mu,\mu,\mu')$.
\end{proof}
\subsection{Network Law of Large Numbers}\label{lln}
In this section, we prove our key result on convergence of network moments:
\begin{lemma}[Network Law of Large Numbers]\label{lemma:general-moment-bound}
    Let $A,B$ be entrywise non-negative matrices independent of $V$, with bounded row and column sums and whose entries are upper bounded in magnitude uniformly by $O(\frac{1}{d})$. 
    If the following conditions hold:
    \begin{equation}\label{eq:lln-conditions}
        1\lesssim d\lesssim n,\quad \|A\|_F^2\wedge \|B\|_F^2=\Omega\left(\frac{n}{d^2}\vee 1\right),\quad \|A\|_F^2\vee \|B\|_F^2=O\left(\frac{n}{d}\right),
    \end{equation}
    then    $$V^TA^TMBV=\sigma^2(\Tr(A^TB)-\frac{1}{n}\Tr(A)\Tr(B))+O_P(\|A\|_F\wedge \|B\|_F),$$
    and $$V^TA^TMB\varepsilon=O_P(\|A\|_F\wedge \|B\|_F).$$
    If in addition to the conditions (\ref{eq:lln-conditions}), $\Tr(A^TB)-\frac{1}{n}\Tr(A)\Tr(B)=\omega(\|A\|_F\wedge \|B\|_F)$, then 
    $$\frac{V^TA^TMBV}{\sigma^2(\Tr(A^TB)-\frac{1}{n}\Tr(A)\Tr(B))}\rightarrow_P 1.$$
    Finally, if in addition to the conditions (\ref{eq:lln-conditions}), $\Tr(A^TB)=\omega(1)$, then
$$\frac{\varepsilon^TA^TMB\varepsilon}{\sigma_\varepsilon^2\Tr(A^TB)}\rightarrow_P 1.$$
\end{lemma}
To prove this, we state and prove the following lemma:
\begin{lemma}\label{lemma:network-lln}
    Let $A,B\in \mathbb{R}_{\geq 0}^{n\times n}$ be matrices which have uniformly bounded column and row sums and whose entries are uniformly upper bounded in magnitude by $O\left(\frac{1}{d}\right)$. Then
    \begin{align}
        V^TA^TBV=\sigma^2 \Tr(A^TB)+O_P\left(\|A\|_F\wedge \|B\|_F+\sqrt{\frac{\Tr(A^TB)}{d}}\right),\label{lemma:network-lln1}
    \end{align}
    and
    \begin{align}
        1^TA^TBV=O_P(\sqrt{n}).\label{lemma:network-lln2}
    \end{align}
    Furthermore,
    \begin{align}
        V^TA^TP_ZBV&=\sigma^2\frac{1}{n}\Tr(A)\Tr(B)+O_P(1+r_1(A,B)+r_2(A,B)),\label{lemma:network-lln3}
    \end{align}
    where $r_1(A,B)=\frac{1}{\sqrt{n}}(\Tr(A)+\|A\|_F+\Tr(B)+\|B\|_F)$ and $r_2(A,B) = \frac{1}{n}(\Tr(A)\|B\|_F+\Tr(B)\|A\|_F+\|A\|_F\|B\|_F)$, and
    \begin{align}
V^TA^TMB\varepsilon=O_P\left((\|A\|_F\wedge \|B\|_F)+\frac{1}{n}(\Tr(A)+\|A\|_F+\sqrt{n})(\sqrt{n}+\|B\|_F)\right).\label{lemma:network-lln4}
    \end{align}
    Finally, if $\Tr(A^TB)=\omega(1)$,
    \begin{align}
\varepsilon^TA^TMB\varepsilon=\sigma^2_\varepsilon\Tr(A^TB)+O_P\left(\frac{1}{n}(\sqrt{n}+\|A\|_F)(\sqrt{n}+\|B\|_F)\right).\label{lemma:network-lln5}
    \end{align}
\end{lemma}

\begin{proof}[Proof of Lemma \ref{lemma:network-lln}]
We first show (\ref{lemma:network-lln1}). Note that $E[V^TA^TBV]=\sigma^2\Tr(A^TB)$. By the same computations as Lemma A.1 from \citet{lee2002consistency},  we have
\begin{align*}
    &E[(V^TA^TBV)^2]-[E(V^TA^TBV)]^2\\
    ={}&\sigma^2(\Tr((A^TB)^2+(A^TB)^TA^TB))+(3\sigma^2-M_4)\sum_i (A^TB)_{ii}^2,
\end{align*}
where $M_4$ is the fourth moment of $\varepsilon$. It follows that 
\begin{align}
V^TA^TBV&=\sigma^2\Tr(A^TB)+O_P\left(\sqrt{\Tr(A^TBB^TA+A^TBA^TB)+\sum_i (A^TB)_{ii}^2}\right).
\end{align}
We have the upper bound
\begin{align}
    \Tr(A^TBB^TA)&\leq \|A\|_F\|BB^TA\|_F\\
    &\leq \|A\|_F\|BB^T\|_2\|A\|_F\lesssim \|A\|_F^2,
\end{align}
where the last line is due to Lemma \ref{lemma:col-sums}. Alternatively, starting with the step $\Tr(A^TBB^TA)\leq \|B\|_F\|B^TAA^T\|_F$ shows that $\Tr(A^TBB^TA)\lesssim \|B\|_F^2$. The same proof can be used to upper bound $\Tr((A^TB)^2)$, so $\Tr(A^TBB^TA+A^TBA^TB)\lesssim \|A\|_F^2\vee \|B\|_F^2$. The only term that remains to be bounded is $\sum_i (A^TB)_{ii}^2$. We have
\begin{align}
    \sum_i (A^TB)_{ii}^2&\leq \sum_i |(A^TB)_{ii}|\max_i |(A^TB)_{ii}|\\
    &=\Tr(A^TB)\max_i |(A^TB)_{ii}|
\end{align}
using non-negativity of $A$ and $B$. We also have
\begin{align}
    \max_i |(A^TB)_{ii}|&=\max_i (Ae_i)^T(Be_i)\\
    &\leq \max_i |Ae_i|_1|Be_i|_\infty\lesssim\frac{1}{d}
\end{align}
using the bounded column sum assumption and the boundedness condition on the entries. Combining these yields (\ref{lemma:network-lln1}). 

For (\ref{lemma:network-lln2}), we have $E[1^TA^TBV]=0$ and 
\begin{align}
E[(1^TA^TBV)^2]&=E[1^TA^TBVV^TB^TA1]\\
&=\sigma_\varepsilon^2\Tr(1^TA^TBB^TA1)\lesssim n
\end{align}
using the bounded row and column sum assumption. The bound follows from Chebyshev's inequality.

For (\ref{lemma:network-lln3}), note that $Z(Z^TZ)^{-1}Z^T$ is the same if $Z=(1,X)$ or $(1,V)$ where $V=X-\mu 1$ by invariance of projections under invertible linear transformations. Thus we will use the definition $Z=(1,V)$. 

We need to analyze $(V^TA^TZ(Z^TZ)^{-1}Z^TBV)$. By the previous parts of this lemma, $V^TA^TV=\sigma^2\Tr(A)+O_P(\|A\|_F)$ and $1^TA^TV=O_P(\sqrt{n})$. Similarly, $V^TB^TV=\sigma^2\Tr(B)+O_P(\|B\|_F)$ and $1^TB^TV=O_P(\sqrt{n})$. Since the $V_i$ are i.i.d., $\|(Z^TZ)^{-1}\|_2\lesssim \frac{1}{n}$. More specifically, by writing out the inverse, 
\begin{align}
    (Z^TZ)^{-1}&=\frac{1}{nV^TV-(1^TV)^2}\begin{pmatrix}
        V^TV&-1^TV\\
        \cdot & n
    \end{pmatrix}.
\end{align}
The denominator is $n(n\sigma^2+O_P(\sqrt{n}))-O_P(n)=n^2\sigma^2+O_P(n^{3/2})$ by the central limit theorem. All the numerator terms are $O_P(n)$. The lower-right corner is equal to $\frac{1}{n\sigma^2+O_P(\sqrt{n})}$.

Expanding $(V^TA^TZ(Z^TZ)^{-1}Z^TBV)$ as a quadratic form of $(Z^TZ)^{-1}$, we see that this can be written as a sum of four terms, corresponding to each entry in $(Z^TZ)^{-1}$. The term corresponding to the lower-right corner is $\frac{1}{n\sigma^2+O_P(\sqrt{n})}\sigma^4\Tr(A)\Tr(B)$, while the other three terms can be seen to comprise of the entries in $r_1$ and $r_2$. We just need to show that the remainder due to the error in the denominator $O_P(\sqrt{n})$ is negligible. This is true since
\begin{align}
    \frac{1}{n+O_P(\sqrt{n})}-\frac{1}{n}&=\frac{O_P(\sqrt{n})}{n(n+O_P(\sqrt{n}))}=O_P(n^{3/2}),
\end{align}
and
\begin{align}
    \left(\frac{1}{n}-\frac{1}{n+O_P(\sqrt{n})}\right)\sigma^2\Tr(A)\Tr(B)&=O_P\left(\frac{1}{n^{3/2}}\left(\frac{n}{d}\right)^2\right)\\
    &=O_P\left(\frac{\sqrt{n}}{d^2}\right)=O_P\left(\frac{\sqrt{n}}{d}\right)=O_P(\|A\|_F\wedge \|B\|_F),
\end{align}
as desired.

We next show (\ref{lemma:network-lln4}). Since $\varepsilon$ is independent of $V$, the same analysis for (\ref{lemma:network-lln2}) shows that $V^TA^TB\varepsilon$ is zero-mean. The variance is now bounded by $O_P(\Tr(A^TBB^TA))=O_P(\|A\|_F^2\vee \|B\|_F^2)$, since in that proof, $11^T$ is replaced by $\varepsilon \varepsilon^T$. This shows that $V^TA^TB\varepsilon=O_P(\|A\|_F\wedge \|B\|_F)$. 

To show the bound on $V^TA^TP_ZB\varepsilon$, similar to the proof of the previous part, we have $\|V^TA^TZ\|_2=O_P(\Tr(A)+\|A\|_F+\sqrt{n})$ and $\|(Z^TZ)^{-1}\|_2=O_P\left(\frac{1}{n}\right)$. We also have $1^TB\varepsilon=O_P(\sqrt{n})$ by the same argument as for $1^TBV$, and $V^TB\varepsilon=O_P(\|B\|_F)$ by a straightforward computation of the variance. Combining yields the desired bound.

Finally, we show the required bound (\ref{lemma:network-lln5}) on $\varepsilon^TA^TMB\varepsilon$. Using the first part of this lemma and substituting $\varepsilon$ for $V$, we see 
$$\varepsilon^TA^TB\varepsilon=\sigma_\varepsilon^2\Tr(A^TB)+O_P(\|A\|_F\wedge \|B\|_F).$$
We just need to upper bound $\varepsilon^TA^TZ(Z^TZ)^{-1}Z^TB\varepsilon$. Note that $\varepsilon^TA^TZ=(\varepsilon^TA^T1,\varepsilon^TA^TV)$. By the previous computations, these are $O_P(\sqrt{n})$ and $O_P(\|A\|_F)$ respectively. The bound follows from the same argument as above.
\end{proof}

We can now use this lemma to prove Lemma \ref{lemma:general-moment-bound}:
\begin{proof}[Proof of Lemma \ref{lemma:general-moment-bound}]
    For the first part, note that Lemma \ref{lemma:network-lln} implies
    \begin{align}
V^TA^TMBV=&\sigma^2\left(\Tr(A^TB)-\frac{\Tr(A)\Tr(B)}{n}\right)+\\
&O_P\left(\|A\|_F\wedge \|B\|_F+\sqrt{\frac{\Tr(A^TB)}{d}}+1+r_1+r_2\right),
    \end{align}
where $r_1$ and $r_2$ are as defined in Lemma \ref{lemma:network-lln}. It suffices to show that $\sqrt{\frac{\Tr(A^TB)}{d}}+1+r_1+r_2=O_P(\|A\|_F\wedge \|B\|_F)$. We have
\begin{align}
    \sqrt{\frac{\Tr(A^TB)}{d}}&\leq \sqrt{\frac{\|A\|_F\|B\|_F}{d}}\lesssim \sqrt\frac{\sqrt{n}\|B\|_F}{d^{3/2}}\\
    &\leq \sqrt{\frac{\sqrt{n}\|B\|_F}{d}}\lesssim \sqrt{\|B\|_F^2}=\|B\|_F,
\end{align}
where the inequalities are using the lower and upper bounds assumed on $\|A\|_F$ and $\|B\|_F$. Reversing the roles of $A$ and $B$ shows that this is also upper bounded in order by $\|A\|_F$.

The bound $1\lesssim \|A\|_F\wedge \|B\|_F$ holds due to the lower bound assumption on the norms $\|A\|_F$ and $\|B\|_F$. For $r_1$, we need to bound $\frac{1}{\sqrt{n}}(\Tr(A)+\|A\|_F)$ and similarly for $B$. Since $A$'s entries are bounded uniformly by $O\left(\frac{1}{d}\right)$, $\Tr(A)\lesssim \frac{n}{d}$. This means
\begin{align}
    \frac{1}{\sqrt{n}}\Tr(A)\lesssim \frac{\sqrt{n}}{d}\leq \|A\|_F\wedge \|B\|_F
\end{align}
by the assumptions of the lemma. For the other term, 
\begin{align}
    \frac{1}{\sqrt{n}}\|A\|_F\lesssim \frac{1}{\sqrt{d}}\leq \frac{1}{\sqrt{d}}\sqrt\frac{n}{d}\lesssim \|A\|_F\wedge \|B\|_F
\end{align}
from the assumed bounds in the lemma. The same analysis symmetrically applied to $B$ yields the bound for $r_1$.

For $r_2$, we have
\begin{align}
    \frac{1}{n}\Tr(A)\|B\|_F\lesssim \frac{1}{d}\|B\|_F.
\end{align}
Clearly, this is at most $\|B\|_F$. It is also at most $\|A\|_F$ since
\begin{align}
    \frac{1}{d}\|B\|_F\lesssim \frac{1}{d}\sqrt{\frac{n}{d}}=\frac{\sqrt{n}}{d^{3/2}}\leq \frac{\sqrt{n}}{d}\lesssim \|A\|_F\wedge \|B\|_F.
\end{align}
The final term to be bounded is
\begin{align}
    \frac{1}{n}\|A\|_F\|B\|_F\lesssim \frac{1}{\sqrt{nd}}\|B\|_F\leq \|B\|_F,
\end{align}
and the same proof with the roles of $A$ and $B$ interchanged shows that this is at most $\|A\|_F$ as well. This proves the first part of the lemma. The second part follows from the bounds on $V^TA^TMB\varepsilon$ proved in Lemma \ref{lemma:network-lln}, since
\begin{align}
    \frac{1}{n}(\Tr(A)+\|A\|_F+\sqrt{n})(\sqrt{n}+\|B\|_F)&\lesssim \frac{1}{\sqrt{n}}(\Tr(A)+\|A\|_F+\sqrt{n})\\
    &\lesssim \frac{1}{\sqrt{n}}\left(\frac{n}{d}+\sqrt{\frac{n}{d}}+\sqrt{n}\right)\\
    &\leq 1\vee \frac{\sqrt{n}}{d}=O(\|A\|_F\wedge \|B\|_F).
\end{align}
Finally, the last part is immediate from the bounds on $\varepsilon^TA^TMB\varepsilon$ in Lemma \ref{lemma:network-lln}, using $\sqrt{n}+\|A\|_F\lesssim \sqrt{n}+\sqrt\frac{n}{d}\lesssim \sqrt{n}$ by assumption on $d$.
\end{proof}
\subsection{OLS Analysis}
\subsubsection{OLS Proofs}
We state the following lemma, which shows that the LLN we provide is satisfied for various relevant matrices.
\begin{lemma}\label{lemma:lln-satisfied}
    If $\rho\geq 0$, the matrices $G$, $GS^{-1}$, and $GH$ satisfy the conditions of the first part of Lemma \ref{lemma:general-moment-bound}. If $d\ll \sqrt{n}$, then $G^2$ also satisfies the conditions. Furthermore, if $\beta\neq 0$, then $\Tr(G^TGH)$ is $\omega(\|GH\|_F)$ and $\Tr((GH)^TGH)$ is $\Omega\left(\frac{n}{d}\right)$, so that the second part of the theorem holds for $(A,B)=(G,GH)$ and $(GH,GH)$.
\end{lemma}
\begin{proof}
    We have $\Tr(G^TG)=\sum_i\sum_j G_{ij}^2=\sum_i d_i\frac{1}{d_i^2}=\Theta(n/d)$. For $GS^{-1}$, we have
    \begin{align}
        \|GS^{-1}\|_F^2&=\Tr(S^{-T}G^TGS^{-1})\geq \Tr(G^TG)
    \end{align}
    since $\rho$ is non-negative. The same proof suffices for $GH$.

    For $G^2$, Proposition \ref{prop:rates} shows that $\|G^2\|_F^2\gtrsim \frac{n}{d^2}$, which is at least $\Omega(1)$ if $d\ll \sqrt{n}$.

    Finally, $\Tr(G^TGH)\geq \Tr(G^TG)$ by non-negativity of $\rho$, and similar for $\Tr((GH)^TGH)$, which shows that both are $\Omega(\Tr(G^TG))=\omega(\|G\|_F)$.
\end{proof}
\begin{proof}[Proof of Lemma \ref{lemma:nonzero-moments}]
    The proofs for $\Tr(G^TG)$, $\Tr(G^TGS^{-1})$, $\Tr((GS^{-1})^TGS^{-1})$, $\Tr(GS^{-1})$ all follow from non-negativity of $\rho$, as in the proof of Lemma \ref{lemma:lln-satisfied}. The same holds for $m_{G,GH}$ and $m_{GH,GH}$ when $\beta\neq 0$, since the sum of the traces is lower bounded by $\beta \Tr(G^TG)$ or $\beta^2 \Tr(G^TG)$, respectively.
\end{proof}
\begin{proof}[Proof of Proposition \ref{cor1}]
    We can write
    \[
\left(\frac{1}{\|G\|_F^2}\tilde{W}_2^T\tilde{W}_2\right)=\frac{1}{\|G\|_F^2}\begin{pmatrix}
    e_1(G,G)&(e_1(G,GH)+v(G,GS^{-1}))\\
    \cdot & (e_1(GH,GH)+v'(GH,GS^{-1})+e_2(GS^{-1},GS^{-1}))
\end{pmatrix}.
\]
Using Lemma \ref{lemma:lln-satisfied}, we see that all the terms of the form $e_1$ or $e_2$ converge. From Lemma \ref{lemma:general-moment-bound}, we have probability bounds on the $v$'s, which are of the order $\|G\|_F$; for example, from the last bound in Lemma \ref{lemma:general-moment-bound}, $$v(G,GS^{-1})=O_P\left(\|G\|_F\wedge \|GS^{-1}\|_F+\frac{1}{n}(\|G\|_F+\sqrt{n})(\sqrt{n}+\|GS^{-1}\|_F)\right),$$
and we can see this is $O_P(\|G\|_F)$ using the same calculations as in Lemma \ref{lemma:lln-satisfied}. Similarly, since $\Tr(GS^{-1})=\Omega(G^2)=\Omega(\|G\|_F^2)=\omega(1)$, the last statement of Lemma \ref{lemma:general-moment-bound} shows the desired convergence result.
\end{proof}
\begin{proof}[Proof of Proposition \ref{prop:ols-clt}]
    The result is immediate from the Central Limit Theorem, using the fact that $G$ and $GH$ have bounded column sums, and the Cramer-Wold device, as in the proof of Lemma A.2 in \citet{lee2002consistency}.
\end{proof}
\subsubsection{Cases where OLS is consistent}\label{dis:ols-consistency}
From Theorem \ref{thm:ols-formal}, we can see that there are some edge cases where the OLS estimator could be consistent. Notably, if $\rho=0$, then there is no endogenous peer effect and thus there is no endogeneity in the model. Naturally, OLS is consistent in this case, as is evident from the theorem by noting that $\Tr(GS^{-1})=\Tr(G)=0$, and so $m_{GS^{-1}}=0$ as well. By examining the proof, we can further see that 
\begin{align}
    \frac{1}{\|G\|_F^2}e_2(I,GS^{-1})&=\frac{1}{\|G\|_F^2}e_2(I,G)\\
    &=O_P\left(\frac{1}{\|G\|_F^2}\|G\|_F\right)=O_P\left(\sqrt\frac{d}{n}\right).
\end{align}
This implies that the bias shrinks at a rate $O_P\left(\sqrt\frac{d}{n}\right)$, which is the same order of the standard deviation. Thus the overall rate of convergence is $O_P\left(\sqrt\frac{d}{n}\right)$. A similar analysis holds for the contextual effect in this case.

If $\rho\neq 0$, the endogenous peer effect estimate is inconsistent. Interestingly, if $\beta =0$ and hence there is no direct covariate effect, then the contextual effect estimate could be consistent. In this case, 
\begin{align}
    \Tr(G^TGH)&=\Tr(G^TG^2S^{-1})\\
    &= \Tr(G^TG^2)+\rho \Tr(G^TG^3)+\rho^2\Tr(G^TG^4S^{-1})\\
    &\leq \Tr(G^TG^2)+|\rho|\|G^TG\|_F\|G^2\|_F+\rho^2 \|G^TG\|_F\|G^2G^2S^{-1}\|_F\\
    &\leq \Tr(G^TG^2)+|\rho|\|G^TG\|_F\|G^2\|_F+\rho^2 \|G^TG\|_F\|G^2\|_F\|G^2S^{-1}\|_2,
\end{align}
using properties of trace and Frobenius norm. Since we have the bounds $\|G^2S^{-1}\|_2=O(1)$, $\|G^TG\|_F=\Theta(\|G^2\|_F)$, and $\Tr(G^TG^2)=\Theta(\Tr(G^3))$ by Lemmas \ref{lemma:order-equivalence} and \ref{lemma:col-sums}, this means $\Tr(G^TGH)=O(\Tr(G^3)+\|G^2\|_F^2)$. If this is $o\left(\frac{n}{d}\right)$, then the bias will vanish. This happens when the network has a small number of triangles and four-cycles, as can be observed by writing
$$\Tr(G^k)=\sum_{c\in C_k}\prod_{(u,v)\in c}\frac{1}{d_u}=\Theta_P\left(\frac{|C_k|}{d^k}\right),$$
where $C_k$ is the set of $k$-cycles. In particular, for random graphs, $\Tr(G^3)=\Theta(1)$ and $\Tr(G^4)=\Theta(\frac{n}{d^2}\vee 1)$ by considering the number of cycles of three or four distinct nodes. It follows that $\Tr(G^TGH)=O\left(\frac{n}{d^2}\vee 1\right)=o\left(\frac{n}{d}\right)$ for $1\lesssim d\lesssim n$. 

This shows that if $\rho \neq 0$ but $\beta=0$, then for random graphs, $\Tr(G^TGH)-\frac{1}{n}\Tr(G)\Tr(GH)=O\left(\frac{n}{d^2}\vee 1\right)$, and by Lemma \ref{lemma:general-moment-bound}, $e_1(G,GH)=V^TG^TMGH=O_P\left(\left(\frac{n}{d^2}\vee 1\right)+\|GH\|_F\right)=O_P\left(\frac{n}{d^2}\vee 1\right)$. The bias is then seen to decay at a rate $O\left(\frac{d}{n}\left(\frac{n}{d^2}\vee 1\right)\right)=O\left(\frac{1}{d}\vee \frac{d}{n}\right)$. 
\subsection{2SLS Analysis}
\subsubsection{Discussion of Assumption \ref{assumption:network-mom2}}\label{wi-app-assumption}
To guide intuition for this assumption, consider the case $\Tr((G^2)^TG^2)-\frac{1}{n}\Tr(G^2)^2$ for a $d$-regular graph. In this case, the first term is $\frac{1}{d^4}\Tr(A^4)=\frac{|C_4|}{d_4}$, while the second term is $\frac{1}{nd^4}\Tr(A^2)^2=\frac{1}{nd^4}|C_2|^2$. The number of 2-cycles is the number of oriented edges, i.e., $nd$, so $\frac{1}{nd^4}\Tr(A^2)^2=\frac{n}{d^2}$. On the other hand, the number of 4-cycles is at least $2nd(d-1)$, which can be seen by considering each vertex to be either the beginning of a a length two path (corresponding to a cycle $i\rightarrow j\rightarrow k\rightarrow j\rightarrow i$), or the middle of one (corresponding to $i\rightarrow j\rightarrow i\rightarrow k\rightarrow i$). This means $\frac{1}{d^4}\Tr(A^4)$ is asymptotically at least twice as large as $\frac{1}{nd^4}\Tr(A^2)^2$; that is, the partialled out term is a constant factor of the original signal.
\subsubsection{2SLS Proofs}
\begin{proof}[Proof of Proposition \ref{prop:2sls-covariance}]
$\|G^2\|_F(\tilde{Z}_2^T\tilde{W}_2)^{-1}F^{-1}$ can be expanded as
    \begin{align}
\|G^2\|_F\frac{1}{\phi}\begin{pmatrix}
\|G\|_F(e(G^2,GH)+v(G^2,GS^{-1}))&-\|G^2\|_F(e(G,GH)+v(G,GS^{-1}))\\
    -\|G\|_Fe(G^2,G)&\|G^2\|_Fe(G,G)
\end{pmatrix},
\end{align}
where $\phi=e(G,G)(e(G^2,GH)+v(G^2,GS^{-1}))-e(G^2,G)(e(G,GH)+v(G,GS^{-1}))$. By Assumption \ref{assumption:triangles}, $e(G^2,G)$ is of a smaller order than $e(G^2,GH)$, since the latter contains terms of the order $G^4$. This means that $\phi$ is of the order $e(G,G)e(G^2,GH)=\Omega(\|G\|_F^2\|G^2\|_F^2)$. Keeping only the higher order terms $e(G,G)e(G^2,GH)$ and removing the lower order $v$ terms in the numerator, we have 
\begin{align}
    \|G^2\|_F(\tilde{Z}_2^T\tilde{W}_2)^{-1}F^{-1}\sim\frac{\|G^2\|_F}{e(G,G)e(G^2,GH)}\begin{pmatrix}
\|G\|_F(e(G^2,GH))&\|G^2\|_F(-e(G,GH))\\
    -\|G\|_Fe(G^2,G)&\|G^2\|_Fe(G,G)
    \end{pmatrix}.
\end{align}
Note that $e(G^2,GH)$ satisfies the LLN, as shown in Lemma \ref{lemma:lln-satisfied}. Consequently, $\frac{\|G^2\|_F}{e(G^2,GH)}\rightarrow_P \frac{1}{\sigma^2m'_{G^2,GH}}$, as defined by Assumption \ref{assumption:network-mom2}. Thus the second column converges to the desired population moment. The top left entry converges by our Assumption \ref{assumption:network-mom2}, and the bottom left entry converges to zero by Assumption \ref{assumption:triangles} and Lemma \ref{lemma:triangle-order}.
\end{proof}
\subsection{Extension to negative peer effects}
In order to extend this analysis to negative peer effects, we would need to ensure that certain quantities, such as $\|GS^{-1}\|_F$, are lower bounded in magnitude by $\Omega(\|G\|_F)$. When peer effects are positive, this is easy to show, since $G$ is entrywise non-negative and $\Tr((GS^{-1})^TGS^{-1})$ is lower bounded by just taking the $\Tr(G^TG)$ term. If peer effects are negative, one possible approach is to write
$$GS^{-1}=G+\rho G^2+\rho^2 G^3+...=G+\rho G^2S^{-1},$$
and use the reverse triangle inequality to bound
\begin{align}
    \|GS^{-1}\|_F-\|G\|_F&\leq \|\rho G^2S^{-1}\|_F\\
    &\leq |\rho| \|G\|_F\|GS^{-1}\|_2\leq |\rho|\|G\|_F\|G\|_2\|S^{-1}\|_2.
\end{align}
Using sub-additivity and sub-multiplicativity of the spectral norm, it is possible to use the series expansion of $S^{-1}$ to argue
$$\|S^{-1}\|_2\leq \frac{1}{1-|\rho| \|G\|_2},$$
valid when $|\rho|\|G\|_2< 1$. Thus we have the bound
$$\|GS^{-1}\|_F-\|G\|_F\leq \frac{|\rho|\|G\|_2}{1-|\rho|\|G\|_2}\|G\|_F.$$
To argue that $\|GS^{-1}\|_F$ and $\|G\|_F$ have the same order, it is sufficient to have
$$\frac{|\rho|\|G\|_2}{1-|\rho|\|G\|_2}\leq 1-\varepsilon $$
for any $\varepsilon >0$. Rearranging, it is sufficient for
$$|\rho|\|G\|_2\leq \frac{1-\varepsilon}{2-\varepsilon}.$$
For regular graphs, $\|G\|_2= 1$, so it is sufficient that $\rho>-\frac{1}{2}$. However, in general, $\|G\|_2$ can grow with the column sums of $G$; while we assume these column sums are uniformly bounded, the condition from this proof technique would become stronger and restrict the possible set of $\rho$ for our results. We leave it to future work to determine weaker conditions that allow these results to be extended to negative values of $\rho$.
\end{appendix}



\end{document}